\documentclass[reqno]{amsart} 

\usepackage{amssymb,amsfonts, amsmath, amsthm, graphicx, enumerate, upgreek}
\usepackage{mathtools, etoolbox, bm}
\usepackage[table]{xcolor}
\usepackage{tikz}
\usepackage[all,arc]{xy}
\usepackage[utf8]{inputenc}
\usepackage[english]{babel}

\usepackage[hidelinks]{hyperref} 

\newtheorem{thm}{Theorem}[section]
\newtheorem*{thm*}{Theorem}

\newtheorem{cor}[thm]{Corollary}
\newtheorem*{cor*}{Corollary}
\newtheorem{prop}[thm]{Proposition}
\newtheorem*{prop*}{Proposition}
\newtheorem{lem}[thm]{Lemma}
\newtheorem*{lem*}{Lemma}
\newtheorem{conj}[thm]{Conjecture}
\newtheorem*{conj*}{Conjecture}
\newtheorem{quest}[thm]{Question}
\newtheorem*{quest*}{Question}

\theoremstyle{definition}
\newtheorem{defn}[thm]{Definition}
\newtheorem*{defn*}{Definition}

\newtheorem{exmp}[thm]{Example}

\theoremstyle{remark}
\newtheorem{rem}[thm]{Remark}

\newtheorem*{ack*}{Acknowledgements}

\newcommand{\stmodsf}{\underline{\mathsf{mod}}\ \!} 
\newcommand{\modsf}{\mathsf{mod}\ \!} 
\newcommand{\Modsf}{\mathsf{Mod}\ \!} 

\newcommand{\N}{\mathbb{N}} 
\newcommand{\Z}{\mathbb{Z}} 

\newcommand{\m}{\mathfrak{m}} 
\newcommand{\depth}{{\rm depth}\ \!} 
\newcommand{\dell}{{\rm dell}\ \!} 
\newcommand{\grade}{{\rm grade}\ \!} 
\newcommand{\pdim}{{\rm projdim}\ \!} 
\newcommand{\idim}{{\rm injdim}\ \!} 
\newcommand{\fdim}{{\rm flatdim}\ \!} 
\newcommand{\findim}{{\rm findim}\ \!} 
\newcommand{\Findim}{{\rm Findim}\ \!} 

\newcommand{\rad}{\mathsf{rad}\ \!} 

\title[The depth, the delooping level and the finitistic dimension]{The depth, the delooping level \\ and the finitistic dimension} 
\author{Vincent G\'elinas}
\address{vincent.gelinas@mail.utoronto.ca} 
\date{}

\begin{document}

\begin{abstract} 
	We investigate two invariants of Noetherian semiperfect rings, namely the depth and a new invariant we call the ``delooping level''. These give lower and upper bounds for the finitistic dimension, respectively. As first theorems, we give a necessary and sufficient criterion for the delooping level to be finite in terms of the splitting of the unit map of a related adjunction, and use this to give a sufficient torsionfreeness criterion for finiteness. We further relate these invariants to the Auslander--Bridger grade conditions for modules, which are vanishing conditions on double Ext duals. As main theorem, we prove that these bounds agree whenever the first non-trivial grade conditions are satisfied for simple modules, so that either invariant computes the finitistic dimension in this case. Over Artinian rings, we show that the delooping level also bounds the big finitistic dimension, and we obtain a sufficient cohomological criterion for the first finitistic dimension conjecture to hold. Over commutative local Noetherian rings, these conditions always hold and we obtain a new characterisation of the depth as the delooping level of the ring. 
\end{abstract} 

\maketitle 
\tableofcontents
\setlength{\parskip}{5pt}

\section*{Introduction}\label{introduction} 
The finitistic dimension is a central invariant in the homological theory of both noncommutative Artin algebras and commutative local Noetherian rings. However, our knowledge of this invariant varies drastically between the two settings. 

Over Artin algebras $\Lambda$, finiteness of the finitistic dimension implies many long standing homological conjectures (generalised Nakayama conjecture, Auslander--Reiten conjecture, Auslander--Gorenstein conjecture, Gorenstein symmetry conjecture, Wakamatsu tilting conjecture). The finiteness of the finitistic dimension itself remains an open problem and is known as the second finitistic dimension conjecture. One also distinguishes the finitistic dimension $\findim \Lambda$, or little finitistic dimension, from the big finitistic dimension $\Findim \Lambda$: the former is the supremum of projective dimensions taken over finitely presented modules of finite projective dimension, while the latter is taken over all modules of finite projective dimension. The first finitistic dimension conjecture for Artin algebras then states that $\findim \Lambda = \Findim \Lambda$. 

Over commutative local Noetherian rings $R = (R, \m, k)$ the finitistic dimensions are well-understood in terms of basic invariants: the Auslander--Buchsbaum formula shows the little finitistic dimension equals the depth of the ring, and the big finitistic dimension is equal to the Krull dimension by work of Bass \cite{Bass62} and Gruson--Raynaud \cite{GR71}. Both dimensions are then finite and bounded above by $\dim_k \m/\m^2$, and the analog of the first finitistic dimension conjecture holds if and only if the ring is Cohen-Macaulay; in particular it holds in the Artinian case.  

By contrast, Huisgen-Zimmermann showed that the first finitistic dimension conjecture admits counterexamples over noncommutative Artin algebras \cite{HZ92}. Additionally, while the natural extension of the depth to Artin algebras $\Lambda$ satisfies the inequality $\depth \Lambda \leq \findim \Lambda^{op}$ due to Jans \cite{Ja61}, equality fails to hold in general, and the difference between the two values can be arbitrarily large. Trying to understand this break in behavior was the motivation for this article. 

In this paper, we make the observation that the equality $\depth R = \findim R$ over commutative local Noetherian rings actually conflates together two invariants: the depth of $R$, and an invariant we call the delooping level of $R$ and denote $\dell R$. Both invariants make sense over general Noetherian semiperfect rings $\Lambda$, and satisfy inequalities 
\begin{align*}
	\depth \Lambda \leq \findim \Lambda^{op} \leq \dell \Lambda. 
\end{align*} 
When $\Lambda$ is Artinian, these are strenghtened to 
\begin{align*}
	\depth \Lambda \leq \findim \Lambda^{op} \leq \Findim \Lambda^{op} \leq \dell \Lambda. 
\end{align*}

The commutative case is then one particular situation when equalities occur: 

\begin{thm*}[Thm. \ref{commutativelocalnoeth}] Let $R$ be a commutative local Noetherian ring. We have equalities 
	\begin{align*}
		\depth R = \findim R = \dell R. 
	\end{align*} 
\end{thm*}
In the general case, equality is controlled by certain double Ext modules. Letting $S \in \modsf \Lambda$ be a simple module, define the index 
\begin{align*}
j_S := \inf \{ i \geq 0 \ | \ {\rm Ext}^i(S, \Lambda) \neq 0 \} \in \Z \cup \{ \infty \} 
\end{align*}
and consider the Ext-vanishing conditions 
\begin{align}\label{Extvanishing}
	{\rm Ext}^j({\rm Ext}^{j_S}(S, \Lambda), \Lambda) = 0 \textup{ for } j < j_S. \tag{$*$} 
\end{align}
We take these to hold vacuously when $j_S = \infty$. Conditions of this type were first considered by Bass \cite{Bass63} and Auslander--Bridger \cite{AB69}. The conditions (\ref{Extvanishing}) are easily seen to hold for simple modules over commutative Noetherian rings, and the above theorem is then a special case of our main theorem: 

\begin{thm*}[Thm. \ref{theoremthree}] Let $\Lambda$ be a Noetherian semiperfect ring, and assume simple modules satisfy conditions \textup{(\ref{Extvanishing})}. Then we have equalities 
\begin{align*}
	\depth \Lambda = \findim \Lambda^{op} = \dell \Lambda. 
\end{align*} 
If $\Lambda$ is furthermore Artinian, then the stronger equalities hold: 
\begin{align*}
	\depth \Lambda = \findim \Lambda^{op} = \Findim \Lambda^{op} = \dell \Lambda. 
\end{align*} 

\end{thm*}  

In the case that $\Lambda = R$ is local commutative, this theorem gives a new proof of the equality $\depth R = \findim R$, independent of the Auslander--Buchsbaum formula. The full result then extends such Auslander--Buchsbaum type equalities to the general setting. In the noncommutative Artinian case, the equality $\findim \Lambda^{op} = \Findim \Lambda^{op}$ also provides a positive case of the first finitistic dimension conjecture. 

\begin{cor*}[Cor. \ref{firstfinitisticdimensionconjecture}] Let $\Lambda$ be an Artinian ring such that the simple modules satisfy condition \textup{(\ref{Extvanishing})}. Then the first finitistic dimension conjecture holds for $\Lambda^{op}$. 
\end{cor*}
The double Ext modules ${\rm Ext}^j({\rm Ext}^{j_S}(S, \Lambda), \Lambda)$ for $j < j_S$ can then be interpreted as obstructions to the validity of Auslander--Buchsbaum type equalities and to the validity of the first finitistic dimension conjecture. 

\subsection*{The delooping level of a ring} We now rapidly describe the new invariant $\dell \Lambda$, and refer the reader to Section \ref{sectiondepthdelooping} for a complete discussion. We let $\Lambda$ denote a Noetherian semiperfect ring below. 

A common approach to studying the finitistic dimension is through the structure of syzygy modules, see for instance \cite{HZ95} for a detailed discussion. When one can realise a module $M \simeq \Omega N$ as a syzygy module, following the standard analogy with the based loop space functor of algebraic topology we then say that $M$ can be delooped. Our approach to the finitistic dimension can be summarised as looking for ``higher level'' deloopings of simple modules. 

The simplest question we can consider is whether a simple module $S$ can be delooped; this occurs precisely when $S \hookrightarrow \Lambda_\Lambda$ occurs in the socle of $\Lambda$. By a classical theorem of Bass, one in fact knows that $\findim \Lambda^{op} = 0$ if and only if every simple $S$ occurs in the socle of $\Lambda$ (Prop. \ref{basstheorem}), and as a result there are almost always simples that cannot be delooped. One is then lead to look for an analogous condition in the case of higher finitistic dimension. A natural generalisation is to consider whether the $n$-th syzygy $\Omega^n S \simeq \Omega^{n+1} N$ can be delooped $(n+1)$-times for some $n \geq 0$. If this occurs for fixed $n \in \N$ and all simple $S$, then one sees that $\findim \Lambda^{op} \leq n$ (Proposition \ref{inequalities}), reducing to one implication of the theorem of Bass when $n = 0$. Analysing this way of bounding the finitistic dimension, natural examples and methods eventually lead one to include deloopings up to retract in the stable module category $\stmodsf \Lambda$. This then motivates the following (see Def. \ref{deloopinglevel} for precision). 

\begin{defn*}[Def. \ref{deloopinglevel}] The delooping level of $S$ is defined as 
\begin{align*} 
	\dell S := \inf \{ n \geq 0 \ | \ \Omega^{n} S \textup{ is a stable retract of } \Omega^{n+1} N \textup{ for some } N \in \stmodsf \Lambda\}.
\end{align*} 
($\dell S = \infty$ if no such $n$ exists). The delooping level of the ring $\Lambda$ is then defined as $\dell \Lambda := \sup_S \dell S$, where $S$ runs over the finitely many simples in $\modsf \Lambda$.
\end{defn*} 

The invariant $\dell S$ seems to interpolate between other invariants which have been used to bound the finitistic dimension, such as the repetition index \cite{GHZ98} or the G-dimension \cite{AB69}. However $\dell S$ can be finite even when the latter are infinite. 

To make any use of this invariant, one wants methods to bound the delooping level from above; for example the main theorem (Thm. \ref{theoremthree}) will follow as soon as one manages to bound $\dell \Lambda \leq \depth \Lambda$. The main tool we employ is the series of adjoint pairs $(\Sigma^n, \Omega^n)$ on the stable module category for $n \geq 0$, where the left adjoint $\Sigma$ to the syzygy functor $\Omega$ was introduced by Auslander--Reiten. Such adjoint pairs give rise, for each module $M$ in the stable category of $\Lambda$, to two approximation towers 
\begin{align*}
	M \to \Omega \Sigma M \to \Omega^2 \Sigma^2 M \to \cdots \to \Omega^n \Sigma^n M \to \Omega^{n+1} \Sigma^{n+1} M \to \cdots	 
\end{align*}
and
\begin{align*}
	\cdots \to \Sigma^{n+1} \Omega^{n+1} M \to \Sigma^n \Omega^n M \to \cdots \to \Sigma^2 \Omega^2 M \to \Sigma \Omega M \to M. 
\end{align*}

These were first studied in \cite{AB69}, and are useful in studying the delooping level. As first method to control the delooping level, we offer a universal characterisation in terms of the above adjoint pairs.  

\begin{thm*}[Thm. \ref{theoremone}, Cor. \ref{corollaryone}] Let $\Lambda$ be a Noetherian semiperfect ring. Then $\dell \Lambda \leq n$ if and only if the unit map $\eta_{n+1}: \Omega^n S \to \Omega^{n+1} \Sigma^{n+1} \Omega^n S$ is a split-monomorphism for each simple module $S$. 
\end{thm*} 

This reduces the study of $\dell \Lambda$ to the structure of the modules $\Omega^{n+1} \Sigma^{n+1} \Omega^n S$. As second method, we offer a torsionfreeness criterion inspired from homological algebra over Gorenstein rings. 

\begin{thm*}[Torsionfreeness criterion, Cor. \ref{torsionfreenesscriterion}] Let $\Lambda$ be a Noetherian semiperfect ring. Assume that for each simple $S$, there is an $n_S$ such that $\Sigma^{n_S} \Omega^{n_S} S$ is torsionfree. Then $\dell \Lambda \leq \sup_S n_S < \infty$. 
\end{thm*} 
The main theorem will follow as an application of our torsionfreeness criterion. As third method, and complementing the torsionfreeness criterion, we construct higher analogs of the Auslander--Bridger exact sequence adapted to studying the torsionfree and reflexive properties for modules of the form $\Sigma^n \Omega^n M$. These exact sequences, although formulated slightly differently, also appeared in prior work of Hoshino--Nishida \cite{HN02}. See Section \ref{subsectionhigherAuslanderBridgersequences} for details. 

Together, these methods allow us to compute or upper bound the delooping level of standard classes of algebras. We obtain results for: 
\begin{itemize}
	\item Gorenstein rings, 
	\item local Artinian rings, 
	\item rings of depth zero, 
	\item radical square zero algebras, 
	\item monomial Artin algebras, 
	\item $n$-syzygy finite algebras, 
	\item commutative local Noetherian rings, 
	\item noncommutative local Noetherian rings, 
	\item Auslander-Gorenstein rings, 
	\item rings whose simple modules satisfy the Ext-vanishing conditions (\ref{Extvanishing}). 
\end{itemize} 

\subsection*{Depth, delooping level and Cohen-Macaulayness} Finally, we discuss our new characterisation of the depth of commutative local Noetherian rings. With the language in place, we can state the full theorem in terms of the adjunction $(\Sigma^n, \Omega^n)$: 

\begin{thm*}[Thm. \ref{commutativelocalnoeth}] Let $R = (R, \m, k)$ be a commutative local Noetherian ring. Then we have equalities $\depth R = \findim R = \dell R$. More precisely, when $R$ has depth $d$, the unit map of $(\Sigma^{d+1}, \Omega^{d+1})$ gives rise to a stable equivalence 
	\begin{align*} 
		\Omega^d k \xrightarrow{\simeq} \Omega^{d+1} \Sigma^{d+1} \Omega^d k 
	\end{align*}
	and $d = \depth R$ is the lowest index for which this holds. 
\end{thm*} 
An analogous result holds for noncommutative local Noetherian rings (Thm. \ref{noncommutativelocalnoeth}). The above theorem shows that $\dell \Lambda$ can be considered as an extension of the depth of commutative local Noetherian rings to the noncommutative domain. In the last section we argue that the inequality 
\begin{align*}
\Findim \Lambda^{op} \leq \dell \Lambda
\end{align*}
is a form of ``Cohen-Macaulay'' property for Noetherian semiperfect rings. We show that this inequality characterises Cohen-Macaulay rings amongst commutative local Noetherian rings, and that it holds also for all noncommutative Gorenstein rings and noncommutative Artinian rings (Thm. \ref{CMinequality}). In fact, our results suggest the stronger possibility that equality holds for Artinian rings, and we record this as an open question.  

\begin{quest*}[Qu. \ref{CMconjecture}] Let $\Lambda$ be an Artinian ring. Do we have $\Findim \Lambda^{op} = \dell \Lambda$? 
\end{quest*}

A positive answer would suggest that the delooping level serves as the natural replacement for the depth in the noncommutative setting.

\subsection*{Structure of the paper}
\begin{itemize}
	\item In Section 1, we introduce the depth and delooping level of $\Lambda$. We then study the left adjoint to syzygies in 1.1, and derive various criteria for controlling the delooping level. In 1.2 we consider the delooping level for standard classes of examples. 
	\item In Section 2, we introduce the Auslander--Bridger grade conditions, and use them in 2.1 to prove the main theorem. In 2.2 we construct the higher Auslander--Bridger exact sequences. 
	\item In Section 3, we state and prove our new characterisation of the depth of commutative local Noetherian rings, and derive an analogous result for noncommutative local Noetherian rings. 
	\item In Section 4, we end the paper with discussing the Cohen-Macaulay inequality $\Findim \Lambda^{op} \leq \dell \Lambda$ and suggest that equality holds for Artinian rings. 
	\end{itemize} 

\subsection*{Acknowledgements} The author is thankful to Ben Briggs for interesting conversations on the subject of this paper.  

\subsection*{Conventions and notation} Noetherian and Artinian rings in this paper will always stand for two-sided Noetherian and two-sided Artinian rings. Unless specified otherwise, $\Lambda$ will always denote a Noetherian semiperfect ring. An Artin algebra will mean a ring $\Lambda$ which is a finite $R$-algebra for $R = (R, \m, k)$ a commutative local Artinian ring. The category $\Modsf \Lambda$ stands for the category of all right $\Lambda$-modules, and the category of left modules $\Modsf \Lambda^{op}$ is identified with right modules over the opposite ring. We let $\modsf \Lambda$ stand for the full subcategory of finitely presented right modules, and without further modifier a module will always mean an object of $\modsf \Lambda$. 

While the paper is written in the generality of Noetherian semiperfect rings $\Lambda$, the Noetherian assumption alone is enough for many results and the semiperfect hypothesis is used to simplify the text. All we require is that simple modules form a test class for measuring projective dimension via Tor. The interested reader will find no difficulty in verifying the results in this generality. 


\subsection*{Stable module categories}
The stable module category $\stmodsf \Lambda$ will always mean the category whose objects coincide with $\modsf \Lambda$, and with morphisms given by residue classes of homomorphisms modulo those factoring through projective objects. The syzygy functor $\Omega: \stmodsf \Lambda \to \stmodsf \Lambda$ is defined on the stable module category, and can be computed as the kernel $\Omega(M) = {\rm ker}(P \twoheadrightarrow M)$ of any epimorphism from a projective $P$ in $\modsf \Lambda$; in particular we do not insist that syzygies be computed from projective covers.

Objects in $\stmodsf \Lambda$ can always be considered as objects of $\modsf \Lambda$, well-defined up to direct sum with projectives. To handle possible ambiguity, we shall use implicitly the following well-known lemma on stable module categories: 

\begin{lem} Let $M, N \in \modsf \Lambda$. 
	\begin{enumerate}[1)] 
	\item The following are equivalent: 
		\begin{enumerate}[i)]
		\item $M$ is a retract of $N$ in $\stmodsf \Lambda$. 
		\item $M$ is a retract of $N \oplus P$ in $\modsf \Lambda$ for some projective object $P$. \smallskip
	\end{enumerate}
	\item The following are equivalent: 
		\begin{enumerate}[i)] 
		\item $M$ is isomorphic to $N$ in $\stmodsf \Lambda$.
		\item $M \oplus Q$ is isomorphic to $N \oplus P$ in $\modsf \Lambda$ for some projective objects $P, Q$. 
	\end{enumerate}
\end{enumerate}
\end{lem} 
\begin{proof}
1) Assume i) and let $s: M \leftrightarrows N: \pi$ satisfy $\pi s = {\rm id}_M$ in $\stmodsf \Lambda$. We have a factorisation of ${\rm id}_M - \pi s = \beta \alpha$ in $\modsf \Lambda$ for morphisms $M \xrightarrow{\alpha} P \xrightarrow{\beta} M$ and $P$ a projective object. The morphisms 
	\begin{align*}
		M \xrightarrow{s' = s \oplus \alpha} N \oplus P \xrightarrow{\pi' = \pi + \beta} M 
	\end{align*} 
compose to $\pi' s' = \pi s + \beta \alpha = {\rm id}_M$ and ii) holds. The converse is clear. 

2) Assume i) and let $s: M \leftrightarrows N: \pi$ be inverse isomorphisms in $\stmodsf \Lambda$. By 1), there is a retract $s': M \leftrightarrows N \oplus P: \pi'$ in $\modsf \Lambda$ with notation as above, and thus $M \oplus Q \cong N \oplus P$ for some complement $Q$. We now show that $Q$ is projective.  

Let $\iota_Q: Q \leftrightarrows N \oplus P: \pi_Q$ be the inclusion and projection, so that $\pi_Q \iota_Q = {\rm id}_Q$ and $\iota_Q \pi_Q = {\rm id}_{N \oplus P} - \pi' s'$. Since $s, \pi$ descend to inverse isomorphisms in the stable module category, so do $s', \pi'$ and thus ${\rm id}_{N \oplus P} = \pi's'$ in $\stmodsf \Lambda$. It follows that ${\rm id}_Q = {\rm id}_Q^2 = \pi_Q (\iota_Q \pi_Q) \iota_Q = \pi_Q({\rm id}_{N \oplus P} - \pi' s')\iota_Q = 0$ in $\stmodsf \Lambda$ and $Q$ is projective, and so ii) holds. The converse is again clear. 
\end{proof}
In the first case we say that $M$ is a stable retract of $N$, and in the second case that $M$ is stably equivalent to $N$. We always denote stable equivalence by $M \simeq N$ and isomorphism in $\modsf \Lambda$ by $M \cong N$.

\section{The depth and delooping level of a ring}\label{sectiondepthdelooping} Let $\Lambda$ be a Noetherian semiperfect ring. The little and big finitistic dimensions of $\Lambda$ are respectively
	\begin{align*}
	\findim \Lambda &:= \sup \{ \pdim M \ | \ M \in \modsf \Lambda,\ \pdim M <\infty \} \\
	\Findim \Lambda &:= \sup \{ \pdim M \ | \ M \in \Modsf \Lambda,\ \pdim M <\infty \}. 
	\end{align*}
	We call $\findim \Lambda$ the finitistic dimension of $\Lambda$ for short. More precisely, this defines the right finitistic dimension and we use $\findim \Lambda^{op}$ (resp. $\Findim \Lambda^{op}$) to denote the left finitistic dimension. When $\Lambda = R = (R, \m, k)$ is a commutative local Noetherian ring, we have $\findim R = \depth R$ by the Auslander--Buchsbaum formula and so the finitistic dimension is well-understood in terms of a more basic invariant. This will serve as a natural starting point for our investigation. 

	In the general situation, one still has a natural notion of depth available. Recall that the grade of $M \in \modsf \Lambda$ is defined as 
\begin{align*}
	\grade M := \inf {\rm Ext}^*_\Lambda(M, \Lambda) =  \inf \{ i \geq 0 \ | \ {\rm Ext}^i_\Lambda(M, \Lambda) \neq 0 \}. 	
\end{align*}
In the local commutative case we have $\depth R = \grade k = \inf {\rm Ext}^*_R(k, R)$. In general, we opt to take the supremum over all simples $S$ to define the depth.  

\begin{defn} The depth of $\Lambda$ is defined as $\depth \Lambda := \sup_S \grade S$. 
\end{defn}

We mention that finiteness of $\depth \Lambda$ for Artin algebras is the statement of the generalised Nakayama conjecture. 

Taking inspiration from the Auslander-Buchsbaum formula, one may try to upper bound the finitistic dimension of a Noetherian semiperfect ring $\Lambda$ by its depth. One class of proofs of the formula proceeds by induction on $\depth R$, reducing to the base case of $\depth R = 0$ where $k$ embeds in the socle of $R$; observe that this turns $k \simeq \Omega^1 N$ into a syzygy module for $N = R/k$. This is then used to bound $\findim R \leq 0 = \depth R$.\footnote{The author has first learned such an argument from lecture notes of S. Spiroff for a VIGRE summer school in 2004 at the University of Utah.} 

Extrapolating the argument in this last bound to higher finitistic dimensions, we are lead to introducing a new invariant. Denote by \emph{stable retracts} the retracts in the stable category $\stmodsf \Lambda$. We then define the delooping level of $M \in \modsf \Lambda$ to be 
	\begin{align*} 
		\dell M := \inf \{ i \geq 0 \ | \ \Omega^i M \textup{ is a stable retract of } \Omega^{i+1} N \textup{ for some } N \in \stmodsf \Lambda \}. 
	\end{align*}
	As with the depth, we define an invariant by taking supremum over the simples $S$. 

	\begin{defn}\label{deloopinglevel} The delooping level of $\Lambda$ is defined as $\dell \Lambda := \sup_S \dell S$. 
	\end{defn}
	Unpacking the definition, one sees that $\dell \Lambda \leq n$ if and only if for each simple $S$, the $n$-th syzygy $\Omega^n S$ is a stable retract of $\Omega^{n+1} N_S$ for some $N_S \in \stmodsf \Lambda$. Our first observation is that $\depth \Lambda$ and $\dell \Lambda$ bound the finitistic dimension. 
	\begin{prop}\label{inequalities} Let $\Lambda$ be a Noetherian semiperfect ring. We have inequalities
	\begin{align*}
		\depth \Lambda \leq \findim \Lambda^{op} \leq \dell \Lambda.
	\end{align*}
	If $\Lambda$ is Artinian, then we have the additional bound 
	 \begin{align*}
		 \depth \Lambda \leq \findim \Lambda^{op} \leq \Findim \Lambda^{op} \leq \dell \Lambda. 
	 \end{align*} 
\end{prop}
\begin{proof} The lower bound $\depth \Lambda \leq \findim \Lambda^{op}$ is due to Jans \cite{Ja61} and obtained as follows. Let $P_* \xrightarrow{\simeq} S$ be a projective resolution of a simple $S$ by finitely generated projectives, and consider then consider truncations of the dual complex $P_*^* = {\rm Hom}_\Lambda(P_*, \Lambda)$: 
	\begin{align*}
		0 \to P_0^* \xrightarrow{\delta_1^*} P_1^* \xrightarrow{\delta_2^*} \dots \xrightarrow{\delta_{n-1}^*} P_{n-1}^* \xrightarrow{\delta_n^*} P_n^* \to {\rm coker}(\delta_n^*) \to 0. 
	\end{align*}
	If ${\rm Ext}^*_\Lambda(S, \Lambda) = 0$, the above is a projective resolution of ${\rm coker}(\delta_n^*)$ for each $n \geq 1$ and one sees that ${\rm Ext}^{n}_{\Lambda^{op}}({\rm coker}(\delta_n^*), \Lambda) = S \neq 0$. Therefore $\findim \Lambda^{op} = \infty$. If ${\rm Ext}^*_\Lambda(S, \Lambda) \neq 0$, this argument gives ${\rm Ext}^{i}_\Lambda(S, \Lambda) \neq 0$ for some $i \leq \findim \Lambda^{op}$. In each case it follows that $\grade S \leq \findim \Lambda^{op}$, and so $\depth \Lambda \leq \findim \Lambda^{op}$ holds. 

	To prove the bound $\findim \Lambda^{op} \leq \dell \Lambda$, we may first assume that $\dell \Lambda = d < \infty$. In this case, for each simple $S$ there is an $N_S \in \modsf \Lambda$ such that $\Omega^d S$ is a stable retract of $\Omega^{d+1} N_S$. Let $M \in \modsf \Lambda^{op}$ be a module of finite projective dimension $n$. Since $\Lambda$ is semiperfect, there exists a simple $S \in \modsf \Lambda$ such that ${\rm Tor}_n^{\Lambda}(S, M) \neq 0$. 
	
	Now, ${\rm Tor}^\Lambda_n(-, M)$ is a functor on $\stmodsf \Lambda$ when $n \geq 1$ and therefore sends stable retracts to retracts of abelian groups. If $n > d$, the non-zero abelian group 
	\begin{align*}
		{\rm Tor}^\Lambda_n(S, M) = {\rm Tor}^\Lambda_{n-d}(\Omega^d S, M)
	\end{align*}
	is then a retract of 
	\begin{align*} 
		{\rm Tor}^\Lambda_{n-d}(\Omega^{d+1}N_S, M) = {\rm Tor}^\Lambda_{n+1}(N_S, M) = 0. 
	\end{align*}
	This contradiction shows that $n \leq d$, and so we have shown $\findim \Lambda^{op} \leq \dell \Lambda$.  
	
	Finally, when $\Lambda$ is Artinian, the Baer criterion shows that for any $X \in \Modsf \Lambda$ of finite injective dimension $n$, there is a simple $S \in \modsf \Lambda$ with \mbox{${\rm Ext}^n_\Lambda(S, X) \neq 0$}. Repeating the above argument gives a bound for the big injective finitistic dimension 
	\begin{align*}
		{\rm InjFindim}\ \! \Lambda = \sup \{ \idim X \ | \ X \in \Modsf \Lambda, \idim X < \infty \} \leq \dell \Lambda. 	
	\end{align*}
	By a theorem of Matlis \cite[Thm. 1]{Ma59}, the big right injective finitistic dimension equals the big left flat finitistic dimension ${\rm InjFindim}\ \! \Lambda = {\rm FlatFindim}\ \! \Lambda^{op}$. Since $\Lambda$ is perfect, the latter equals ${\rm Findim}\ \! \Lambda^{op}$ and ${\rm Findim}\ \! \Lambda^{op} = {\rm InjFindim}\ \! \Lambda \leq \dell \Lambda$ follows. 
\end{proof} 

We point out that either inequality in $\depth \Lambda \leq \findim \Lambda^{op} \leq \dell \Lambda$ can be strict. Consider the differences 
\begin{align*}
	&N_1 = \findim \Lambda^{op} - \depth \Lambda	\\
	&N_2 = \dell \Lambda - \findim \Lambda^{op}.	
\end{align*}
Concerning $N_1$, Huisgen-Zimmermann has constructed examples of Artin algebras with radical square zero in which $N_1 \gg 0$ (see remark after Prop. 2 in \cite{Kr98}, and use that $\findim \Lambda^{op} = \Findim \Lambda^{op}$ for radical square zero algebras \cite[Sect. 3]{HZ95}). 

Concerning $N_2$, note that the bound $\findim \Lambda^{op} \leq \Findim \Lambda^{op} \leq \dell \Lambda$ for Artinian rings shows that an equality $\findim \Lambda^{op} = \dell \Lambda$ leads to the equality $\findim \Lambda^{op} = \Findim \Lambda^{op}$, and so to the validity of the first finitistic dimension conjecture for $\Lambda^{op}$. The first counterexamples were also constructed by Huisgen-Zimmermann \cite{HZ92}; this was followed-up by Smal{\o} \cite{Sm98} who constructed Artinian rings of radical cube zero over which $\Findim \Lambda^{op} - \findim \Lambda^{op} \gg 0$, and so $N_2 \gg 0$ over these rings as well. 

The last bound to consider is the bound $\Findim \Lambda^{op} \leq \dell \Lambda$ for Artinian rings. Equality in this case appears more mysterious, and will be considered in Section \ref{sectioncohenmacaulayinequality}. 

\subsection*{A toy example}\label{subsectiontoyexample} At first glance, it isn't clear that $\dell \Lambda$ has any chance of being finite in general. We first consider a simple but important example. 

Recall that $M \in \modsf \Lambda$ is Gorenstein-projective if there exists an unbounded, acyclic complex of finite projectives
\begin{align*}
	\cdots \to C_2 \to C_1 \to C_0 \to C_{-1} \to C_{-2} \to \cdots 
\end{align*}
whose zero-truncation resolves $M$, and such that ${\rm Hom}_\Lambda(C_*, \Lambda)$ is also acyclic. We say that $C_*$ is a complete resolution of $M$. In this case $M \simeq \Omega^n N$ for any $n \geq 0$ and appropriate $N$ obtained by truncating the above complex further to the right, and so $M$ is an arbitrarily high order syzygy module. 

Recall also that $\Lambda$ is Gorenstein if $\idim \Lambda_\Lambda < \infty$ and $\idim \Lambda^{op}_{\Lambda^{op}} < \infty$, in which case these values are equal by a theorem of Zaks \cite[Appendix]{Zaks69}. We refer to the common value $\idim \Lambda_\Lambda = \idim \Lambda^{op}_{\Lambda^{op}}$ as the dimension of $\Lambda$. Over a Gorenstein ring, Buchweitz \cite[4.4]{Bu86} has shown that the Gorenstein-projectives are precisely those modules $M$ satisfying ${\rm Ext}^i_\Lambda(M, \Lambda) = 0$ for $i > 0$, and it follows that $n$-th syzygy modules are Gorenstein-projective as soon as $n \geq \idim \Lambda_\Lambda$. 

\begin{exmp}\label{gorensteinexample} Let $\Lambda$ be Gorenstein of dimension $d$. For any simple module $S$, the $d$-th syzygy module $\Omega^d S$ is Gorenstein-projective and so an arbitrarily high order syzygy. It follows that $\dell S \leq d$, and so the delooping level of $\Lambda$ is at most $d$. 
\end{exmp} 

In fact the delooping level equals precisely the injective dimension for Gorenstein rings, as we will see (Example \ref{fullgorensteinexample}).

\subsection{Adjunction for syzygies and a torsionfreeness criterion}\label{subsectionadjunction} 
Let $\Lambda$ be a Noetherian semiperfect ring and $S$ a simple module. Suppose we suspect that $\dell S \leq n$, so that $\Omega^n S$ should be a stable retract of $\Omega^{n+1} N_S$ for some $N_S$. We are faced with the problem of constructing $N_S$ in some way from $S$. 

The case of $\Lambda$ Gorenstein is illuminating, and we rephrase Example \ref{gorensteinexample} as follows. Buchweitz \cite[5.0]{Bu86} has constructed a Gorenstein-projective approximation 
\begin{align*}
	S^{\sf GP} \to S	 
\end{align*}
meaning that $S^{\sf GP}$ is Gorenstein-projective and any other such map uniquely factors through it in the stable category, with the property that taking $d$-th syzygy induces a stable equivalence 
\begin{align*}
	\Omega^d S^{\sf GP} \xrightarrow{\simeq} \Omega^d S.	
\end{align*}
The syzygy functor is invertible on the stable category of Gorenstein-projectives, and setting $N_S := \Omega^{-1} S^{\sf GP}$ gives stable equivalences 
\begin{align*}
	\Omega^{d+1} N_S \xrightarrow{\simeq} \Omega^d S^{\sf GP} \xrightarrow{\simeq} \Omega^d S.  
\end{align*}

Over general $\Lambda$, it turns out that fragments of the Gorenstein-projective machinery still make sense in $\stmodsf \Lambda$ and these can be used to provide a canonical candidate for $N_S$. The crucial tool is the left adjoint to the syzygy functor constructed by Auslander--Reiten. 

\subsection*{The left adjoint to syzygies}\label{subsectionleftadjoint} Let $\Lambda$ be a Noetherian semiperfect ring throughout. We denote the Auslander--Bridger transpose by ${\rm Tr}: \stmodsf \Lambda \to \stmodsf \Lambda^{op}$. Recall that ${\rm Tr}$ is computed by dualising a presentation by finite projectives 
\begin{align*}
	P_1 \to P_0 \to M \to 0
\end{align*} 
to a presentation 
\begin{align*}
	P_0^* \to P_1^* \to {\rm Tr}M \to 0.	
\end{align*}
The transpose is a contravariant duality in that ${\rm Tr}^2 = {\rm id}$. We then define the suspension functor as $\Sigma := {\rm Tr}\ \! \Omega \ \! {\rm Tr}: \stmodsf \Lambda \to \stmodsf \Lambda$. 

\begin{prop}[{\cite[Cor. 3.4]{AR96}}]The pair $(\Sigma, \Omega)$ is an adjoint pair on $\stmodsf \Lambda$. 
\end{prop} 

The module $\Sigma M$ has a simple presentation in terms of $M$. Let $P_* \xrightarrow{\sim} M$ and $Q_* \xrightarrow{\sim} M^*$ be projective resolutions, and write $\sigma: M \to M^{**}$ for the canonical evaluation map. Following Buchweitz \cite[5.6.1]{Bu86}, we introduce the Norm map of $M$ given as the composition 
\begin{align*} 
\xymatrix@R=5pt{ 
	\cdots \ar[r]& P_1 \ar[r]&  P_0 \ar[dr] \ar[rrr]^-{\textup{Norm}} &&& Q_0^* \ar[r]& Q_1^* \ar[r]& \cdots \\ 
	&&& M \ar[r]^-{\sigma}& M^{**} \ar[ur] 
}
\end{align*}

The next lemma follows from unpacking the definition. 

\begin{lem}\label{cokernorm} Let $M \in \modsf \Lambda$. We have a presentation
	\begin{align*}
		P_0 \xrightarrow{ {\rm Norm}} Q_0^* \to \Sigma M \to 0. 	
	\end{align*}
	In other words, $\Sigma M$ is the cokernel of the Norm map. 
\end{lem} 
Define $tM := {\rm ker}(\sigma)$ to be the torsion part of $M$, and we say that $M$ is torsionfree when $tM = 0$. From the above lemma, we see that 
\begin{align*}
	\Omega \Sigma M \simeq {\rm im}({\rm Norm}: P_0 \to Q_0^*) = {\rm im}(\sigma: M \to M^{**}) \cong M/tM.
\end{align*} 
In particular $\Omega \Sigma M \simeq M$ whenever $M$ is torsionfree. This hints at a close relationship between torsionfreeness and $\Omega\Sigma(-)$; we will see a more canonical statement below. 

\subsection*{Approximation towers and torsionfreeness}\label{subsectionapproximationtowers} 
The adjoints $(\Sigma, \Omega)$ are endofunctors, and taking powers we obtain further adjoint pairs $(\Sigma^n, \Omega^n)$ for each $n \geq 1$. Let us write 
\begin{align*}
&\eta_n: M \to \Omega^n \Sigma^n M \\ 
&\varepsilon_n: \Sigma^n \Omega^n M \to M 
\end{align*}
for the unit and the counit of this adjunction. Associated is an approximation tower 
\begin{align*} 
 M \to \Omega \Sigma M \to \Omega^2 \Sigma^2 M \to \cdots \to \Omega^n \Sigma^n M \to \Omega^{n+1} \Sigma^{n+1} M \to \cdots	
 \end{align*}
 where each map is obtained from the unit $\eta_1: \Sigma^i M \to \Omega \Sigma(\Sigma^i M)$ by applying $\Omega^i(-)$. Dually, the counit map $\varepsilon_1$ gives rise to an approximation tower 
 \begin{align*} 
	 \cdots \to \Sigma^{n+1} \Omega^{n+1} M \to \Sigma^n \Omega^n M \to \cdots \to \Sigma^2 \Omega^2 M \to \Sigma \Omega M \to M
\end{align*}
with dual properties, and the two constructions are exchanged by ${\rm Tr}(-)$. We record a simple but useful property, which follows from standard arguments on adjunctions. 
\begin{lem}\label{unitcounitlemma} Consider the morphisms in the above tower for $M$. 
	\begin{enumerate}[i)]
	\item The composition $M \to \dots \to \Omega^n \Sigma^n M$ agrees with the unit map $\eta_n$. 
	\item The composition $\Sigma^n \Omega^n M \to \dots \to M$ agrees with the counit map $\varepsilon_n$. 
\end{enumerate} 
\end{lem}

These towers were studied in \cite{AB69}, although under the different the notation $D_k^2 = \Omega^k \Sigma^k$ and $J_k^2 = \Sigma^k \Omega^k$, where $D_k = \Omega^k {\rm Tr}$ and $J_k := {\rm Tr} \Omega^k$. They are closely related to torsionfreeness properties of $M$ and ${\rm Tr}M$, as in the following definition. 
\begin{defn} A module $M$ is $n$-torsionfree is ${\rm Ext}^i_{\Lambda^{op}}({\rm Tr}M, \Lambda) = 0$ for $1 \leq i \leq n$. 
	
\end{defn}
From \cite{AB69} (see also Section \ref{subsectionhigherAuslanderBridgersequences}), we know that $M$ is $1$-torsionfree if and only if it is torsionfree, and $M$ is $2$-torsionfree if and only if it is reflexive. The next proposition collects standard properties of these approximation towers.  

\begin{prop}\label{towerproperties} Let $M \in \modsf \Lambda$. The following properties hold in $\stmodsf \Lambda$: 
	\begin{enumerate}[i)] 
	\item $\Omega\Sigma M \simeq M/tM$, and $M \to \Omega\Sigma M$ is represented by the map $M \twoheadrightarrow M/tM$. 
	\item $\Omega^2 \Sigma^2 M \simeq M^{**}$, and $M \to \Omega^2 \Sigma^2 M$ is represented by the map $M \xrightarrow{\sigma} M^{**}$. 
	\item $M \to \Omega\Sigma M$ is a stable equivalence if and only if $M$ is torsionfree. 
	\item $M \to \Omega^2\Sigma^2 M$ is a stable equivalence when $M$ is reflexive. The converse holds over rings over which double dual modules are reflexive. 
	\item If $M$ is $n$-torsionfree then the first $n$ layers are stable equivalences 
		\begin{align*}
			M \xrightarrow{\simeq} \Omega \Sigma M \xrightarrow{\simeq} \cdots \xrightarrow{\simeq} \Omega^n \Sigma^n M. 
		\end{align*}
	\item If ${\rm Tr}M$ is $n$-torsionfree then the first $n$ dual layers are stable equivalences 
 \begin{align*} 
	 \Sigma^{n} \Omega^{n} M \xrightarrow{\simeq} \cdots \xrightarrow{\simeq} \Sigma \Omega M \xrightarrow{\simeq} M. 
\end{align*}
	\item If $M$ is Gorenstein-projective, then so is $\Sigma M$ and we have stable equivalences
		\begin{align*}
			\Sigma \Omega M \xrightarrow{\simeq} M \xrightarrow{\simeq} \Omega \Sigma M. 	
		\end{align*}
\item If $\Lambda$ is Gorenstein of dimension $d$, then for any $M \in \underline{\mathsf{mod}}\ \! \Lambda$ the tower 
\begin{align*}
\cdots \to \Sigma^{n+1} \Omega^{n+1} M \to \Sigma^n \Omega^n M \to \Sigma^{n-1} \Omega^{n-1} M \to \dots \to \Sigma^1 \Omega^1 M \to M
\end{align*} 
collapses to stable equivalences $\Sigma^{n+1} \Omega^{n+1} M \xrightarrow{\simeq} \Sigma^{n} \Omega^n M$ for $n \geq d$ and the map $\Sigma^d \Omega^d M \to M$ is a Gorenstein-projective approximation. 
\end{enumerate} 
\end{prop} 
\begin{proof}This is mostly standard (up to differing language), and we provide proofs or point to the literature as needed. 

Case i). Realising $\Sigma M$ as the cokernel of the Norm map (Lemma \ref{cokernorm}), we saw that $\Omega \Sigma M \simeq M/tM$. The claim about the unit map is easiest to see from the presentation of the adjoint isomorphism given in \cite[Section 3]{AR96}. \bigskip

Case ii). Applying the identity $X^* \simeq \Omega^2 {\rm Tr}X$ for $X \in \modsf \Lambda$ twice gives 
\begin{align*}
	M^{**} \simeq \Omega^2 {\rm Tr} \ \!M^* \simeq \Omega^2 {\rm Tr}\ \! \Omega^2 {\rm Tr}\ \! M = \Omega^2 \Sigma^2 M. 
\end{align*}
The claim about the unit map again is easiest to see from the presentation of the adjoint isomorphism in \cite[Section 3]{AR96}.\bigskip

Case iii). If $M$ is torsionfree then $M \xrightarrow{\simeq} \Omega \Sigma M$ from i). Conversely if $M \xrightarrow{\simeq} \Omega \Sigma M$ is a stable equivalence then $M \simeq \Omega \Sigma M \simeq M/tM$, and there are projective modules $P, Q$ for which $M \oplus P \cong M/tM \oplus Q$. Therefore $M$ is a direct summand of a torsionfree module and so is torsionfree. \bigskip

Case iv). The first implication follows from ii). Conversely, if $M^{**}$ is reflexive and $M \xrightarrow{\simeq} \Omega^2 \Sigma^2 M \simeq M^{**}$ then $M \oplus P \cong M^{**} \oplus Q$ for some projectives $P, Q$. Therefore $M$ is a summand of a reflexive module and so is reflexive. \bigskip

Case v). Since $M$ is $n$-torsionfree, in particular torsionfree, we have $M \xrightarrow{\simeq} \Omega \Sigma M$ by i). We claim that $\Sigma^k M$ is also torsionfree for all $0 \leq k \leq n-1$.  Using ${\rm Tr} \Sigma^k \simeq \Omega^k {\rm Tr}$ gives 
\begin{align*}
	{\rm Ext}^1({\rm Tr} \Sigma^k M, \Lambda) 
	= {\rm Ext}^{k+1}({\rm Tr} M, \Lambda) = 0 
\end{align*} 
since $1 \leq k+1 \leq n$. It follows that $\Sigma^k M \xrightarrow{\simeq} \Omega \Sigma^{k+1} M$ is a stable equivalence, and applying $\Omega^k(-)$ gives $\Omega^k \Sigma^k M \xrightarrow{\simeq} \Omega^{k+1} \Sigma^{k+1} M$ as we wanted. \bigskip

Case vi). This is dual to v). \bigskip

Case vii). Let $P_* \xrightarrow{\simeq} M$ and $Q_* \xrightarrow{\simeq} M^*$ be projective resolutions, and consider the complex $C_*$ constructed by gluing along the Norm map 
\begin{align*} 
\xymatrix@R=5pt{ 
	\cdots \ar[r]& P_1 \ar[r]&  P_0 \ar[dr] \ar[rrr]^-{\textup{Norm}} &&& Q_0^* \ar[r]& Q_1^* \ar[r]& \cdots \\ 
	&&& M \ar[r]^-{\sigma}& M^{**} \ar[ur] 
} 
\end{align*} 
If $M$ is Gorenstein-projective then this is a complete resolution for $M$, and $\Sigma M = {\rm coker}({\rm Norm})$ has complete resolution given by the suspension of $C_*$ as a complex. Hence $\Sigma M$ is Gorenstein-projective. For the second claim, recall that Gorenstein-projectives are characterised by the condition (\cite[Prop. 3.8]{AB69}) 
\begin{align*}
	{\rm Ext}_\Lambda^i(M, \Lambda) = 0 \textup{ and } {\rm Ext}^i_{\Lambda^{op}}({\rm Tr}M, \Lambda) = 0 \textup{ for } i > 0. 
\end{align*}
Hence $M$ and ${\rm Tr}M$ are torsionfree and so $\Sigma \Omega M \xrightarrow{\simeq} M \xrightarrow{\simeq} \Omega \Sigma M$ follows from i). \bigskip

Case viii). When $\Lambda$ is Gorenstein of dimension $d$ then $\Omega^n M$ is Gorenstein-projective for any $n \geq d$. By vii) we have $\Sigma \Omega(\Omega^nM) \xrightarrow{\simeq} \Omega^n M$, and applying $\Sigma^n(-)$ then gives $\Sigma^{n+1}\Omega^{n+1}M \xrightarrow{\simeq} \Sigma^n \Omega^n M$. 

For the second claim, note that $\Sigma^d \Omega^d M$ is Gorenstein-projective. The Gorenstein-projectives over a Gorenstein ring are precisely the modules of the form $\Omega^d X$, and as Gorenstein-projectives are closed under ${\rm Tr}(-)$ these are also precisely the modules of the form $\Sigma^d Y$. By standard properties of adjunctions \cite[Section 3.1]{Bo94}, any morphism from a Gorenstein-projective $f: \Sigma^d Y \to M$ factors uniquely as 
\begin{align*}
\xymatrix{
	\Sigma^d Y \ar[dr]_-{f} \ar[r]^-{\Sigma^d(g)}& \Sigma^d \Omega^d M \ar[d]^-{\varepsilon_d} \\ 
	& M
}
\end{align*}
with $g: Y \to \Omega^d M$ the adjoint morphism to $f$, as we wanted. \end{proof}

\subsection*{Revisiting the Gorenstein case}\label{subsectionrevisitinggorenstein} 
With this machinery in place, we can now revisit the argument bounding the delooping level of a Gorenstein ring given in Example \ref{gorensteinexample}. Let $\Lambda$ be Gorenstein of dimension $d$ and let $S$ be a simple module. By Proposition \ref{towerproperties} viii) the counit map 
\begin{align*} 
	\Sigma^d \Omega^d S \to S 	
\end{align*}
is a Gorenstein-projective approximation, which we saw earlier is sent under $\Omega^d(-)$ to a stable equivalence 
\begin{align*} 
	\Omega^d \Sigma^d \Omega^d S \xrightarrow{\simeq} \Omega^d S. 
\end{align*}
Since $\Sigma^d \Omega^d S$ is Gorenstein-projective, it is torsionfree and so we have a stable equivalence 
\begin{align*}
	\Sigma^d \Omega^d S \xrightarrow{\simeq} \Omega \Sigma(\Sigma^d \Omega^d S). 
\end{align*} 
Combining the two, we obtain a pair of morphisms
\begin{align*}
	S \leftarrow \Sigma^d \Omega^d S \xrightarrow{\simeq} \Omega \Sigma^{d+1} \Omega^d S
\end{align*}
which are sent under $\Omega^d(-)$ to stable equivalences
\begin{align*} 
	\Omega^d S \xleftarrow{\simeq} \Omega^d \Sigma^d \Omega^d S \xrightarrow{\simeq} \Omega^{d+1} \Sigma^{d+1} \Omega^d S. 
\end{align*} 
The left hand side morphism is always a retract in the general case (as in any adjunction), with section given by the unit map $\eta_d$ in the opposite direction; this section is the natural morphism for us in general, and happens to be a stable equivalence only because $\Omega^d S$ is Gorenstein-projective (apply Proposition \ref{towerproperties} v)). We can therefore turn this diagram around to the more canonical 
\begin{align*}
	\Omega^d S \xrightarrow{\simeq} \Omega^d \Sigma^d \Omega^d S \xrightarrow{\simeq} \Omega^{d+1} \Sigma^{d+1} \Omega^d S.
\end{align*} 
The resulting composition is also the unit map $\eta_{d+1}$ by Lemma \ref{unitcounitlemma}. Setting $N_S := \Sigma^{d+1} \Omega^d S$, we obtain the sought-after canonical construction mentioned at the beginning of Section \ref{subsectionadjunction}. 

In this way, we see that the inequality $\dell S \leq d$ over Gorenstein rings of dimension $d$ is a consequence of a deeper structural statement, phrased in terms of the general machinery associated to the adjoint pair $(\Sigma^n, \Omega^n)$. We now use this machinery to formulate a pair of theorems valid over general Noetherian semiperfect rings $\Lambda$. 

\subsection*{The delooping level and a torsionfreeness criterion}\label{subsectiontorsionfreenesscriterion} Recall that we have defined the delooping level of $M \in \stmodsf \Lambda$ as 
\begin{align*}
	\dell M := \inf \{ n \geq 0 \ | \ \Omega^n M \textup{ is a stable retract of } \Omega^{n+1} N \textup{ for some } N \in \stmodsf \Lambda \}
\end{align*} 
and $\dell \Lambda = \sup_S \dell S$, where $S$ runs over the simples. Since $\findim \Lambda^{op} \leq \dell \Lambda$ by Proposition \ref{inequalities} (resp. $\Findim \Lambda^{op} \leq \dell \Lambda$ for Artin algebras), any method to bound the delooping level of simple modules gives a method to prove finiteness of the finitistic dimension. In this section, we provide two such results. 

The first theorem is a universal characterisation for $\dell M$, which follows immediately from the pair $(\Sigma^n, \Omega^n)$ forming an adjunction. We say for short that the unit map $\eta_{n}: X \to \Omega^{n} \Sigma^{n} X$ splits if it is a section in $\stmodsf \Lambda$. 
\begin{thm}\label{theoremone} Let $\Lambda$ be a Noetherian semiperfect ring. The following are equivalent for any $M \in \stmodsf \Lambda$ and $n \geq 0$: 
	\begin{enumerate}[i)] 
	\item $\Omega^n M$ is a stable retract of $\Omega^{n+1} N$ for some $N \in \stmodsf \Lambda$. 
	\item $\Omega^n M$ is a stable retract of $\Omega^{n+1} \Sigma^{n+1} \Omega^n M$. 
	\item The unit map $\eta_{n+1}: \Omega^n M \to \Omega^{n+1} \Sigma^{n+1} \Omega^n M$ splits in $\stmodsf \Lambda$. 
\end{enumerate} 
Therefore the delooping level is equal to 
\begin{align*}
	\dell M = \inf \{n \geq 0 \ | \ \eta_{n+1}: \Omega^n M \to \Omega^{n+1} \Sigma^{n+1} \Omega^n M \textup{ splits in } \stmodsf \Lambda \}. 
\end{align*}
\end{thm}
\begin{proof} Clearly $iii) \implies ii) \implies i)$, and it suffices to show $i) \implies iii)$. We use that any morphism $f: X \to \Omega^{n+1} N$ in $\stmodsf \Lambda$ factors as 
\begin{align*}
	\xymatrix{X \ar[dr]_-{f} \ar[r]^-{\eta_{n+1}}& \Omega^{n+1}\Sigma^{n+1} X \ar[d]^-{\Omega^{n+1}\overline{f}} \\ 
		& \Omega^{n+1} N
} 
\end{align*}
where $\overline{f}: \Sigma^{n+1} X \to N$ corresponds to $f$ under the adjunction \cite[Section 3.1]{Bo94}. Applying this $X = \Omega^n M$ and any section $\iota: \Omega^n M \to \Omega^{n+1} N$ to a retract $\pi$ in $\stmodsf \Lambda$, we obtain a commutative diagram  
\begin{align*}
	\xymatrix{\Omega^n M \ar[dr]_-{\iota} \ar[r]^-{\eta_{n+1}}& \Omega^{n+1}\Sigma^{n+1} \Omega^n M \ar[d]^-{\Omega^{n+1}\overline{\iota}} \\ 
		& \Omega^{n+1} N \ar@/^12pt/[ul]^-{\pi}. 
} 
\end{align*}
Therefore $\eta_{n+1}$ splits as claimed. The last statement is clear. 
\end{proof} 
Theorem \ref{theoremone} has for immediate consequences: 

\begin{cor} Let $\Lambda$ be Noetherian semiperfect and $M \in \stmodsf \Lambda$. Assume that $\eta_{n+1}: \Omega^n M \to \Omega^{n+1} \Sigma^{n+1} \Omega^n M$ splits. Then $\eta_{m+1}: \Omega^m M \to \Omega^{m+1} \Sigma^{m+1} \Omega^m M$ splits for any $m \geq n$. 
\end{cor} 
\begin{proof}
Apply the equivalence $i) \iff iii)$ of Theorem \ref{theoremone}.
\end{proof}
\begin{cor}\label{corollaryone} Let $\Lambda$ be Noetherian semiperfect. The following are equivalent: 
	\begin{enumerate}[i)]
	\item $\dell \Lambda \leq n$. 
	\item $\Omega^n S$ is a stable retract of $\Omega^{n+1} \Sigma^{n+1} \Omega^n S$ for each simple module $S$. 
	\item The unit map $\eta_{n+1}: \Omega^n S \to \Omega^{n+1} \Sigma^{n+1} \Omega^n S$ splits for each simple module $S$. 
\end{enumerate} 
\end{cor} 

As a result deciding the condition $\dell \Lambda \leq n$ becomes tractable, reduced to determining the structure of the modules $\Omega^{n+1} \Sigma^{n+1} \Omega^n S$ for $S$ simple. This suggests such modules should play an interesting role in the study of the finitistic dimension. 

The next theorem encapsulates the ``Gorenstein delooping'' argument laid out in the last section, which now becomes a torsionfreeness criterion valid for general $\Lambda$. 
\begin{thm}\label{theoremtwo} Let $M \in \stmodsf \Lambda$. If $\Sigma^n \Omega^n M$ is torsionfree then $\dell M \leq n$. 
\end{thm} 
\begin{proof} A general property of adjunctions is that the counit map
	\begin{align*}
		\Sigma^n \Omega^n M \xrightarrow{\varepsilon_n} M	 
	\end{align*} 
	becomes a retract after applying $\Omega^n(-)$, with corresponding section the unit map 
	\begin{align*}
		\Omega^n M \xrightarrow{\eta_n} \Omega^n \Sigma^n \Omega^n M. 
	\end{align*}
If $\Sigma^n \Omega^n M$ is torsionfree then we have a stable equivalence
\begin{align*}
	\Sigma^n \Omega^n M \xrightarrow{\simeq} \Omega \Sigma(\Sigma^n \Omega^n M). 
\end{align*}
Applying $\Omega^n(-)$ and combining with the above gives a section
\begin{align*}
	\Omega^n M \xrightarrow{\eta_n} \Omega^n \Sigma^n \Omega^n M \xrightarrow{\simeq} \Omega^{n+1} \Sigma^{n+1} \Omega^n M 
\end{align*}
which Lemma \ref{unitcounitlemma} shows is simply the unit map $\eta_{n+1}$. Therefore $\dell M \leq n$. 
\end{proof}
Applying Theorem \ref{theoremtwo} to simple modules, we obtain a torsionfreeness criterion for finiteness of the delooping level. 
\begin{cor}[Torsionfreeness criterion]\label{torsionfreenesscriterion} Let $\Lambda$ be a Noetherian semiperfect ring. Assume that for each simple $S$, there is an $n_S$ such that $\Sigma^{n_S} \Omega^{n_S} S$ is torsionfree. Then $\dell \Lambda \leq \sup_S n_S < \infty$. 
\end{cor} 

We can rephrase this criterion in special cases of interest. Recall that an Artin algebra $\Lambda$ is a finite $R$-algebra over a local commutative Artinian ring $R = (R, \m, k)$. We let $D := {\rm Hom}_R(-, E(k))$ be the Matlis duality between $\modsf \Lambda$ and $\modsf \Lambda^{op}$, and $\tau = D{\rm Tr}$ be the Auslander--Reiten translate. Using Auslander--Reiten duality, we can rewrite the obstruction to torsionfreeness of $X \in \modsf \Lambda$ as 
\begin{align*}
	{\rm Ext}^1_{\Lambda^{op}}({\rm Tr}X, \Lambda) &\cong {\rm Ext}^1_\Lambda(D\Lambda, \tau X)	\cong D {\rm \underline{Hom}}_\Lambda(X, D\Lambda). 
\end{align*}
Applying this to $X = \Sigma^n \Omega^n M$ and using the adjunction isomorphism
\begin{align*}
	{\rm \underline{Hom}}_\Lambda(\Sigma^n \Omega^n M,\ \! D \Lambda) \cong {\rm \underline{Hom}}_\Lambda(\Omega^n M,\ \! \Omega^n D\Lambda) 
\end{align*}
we conclude that $\Sigma^n \Omega^n M$ is torsionfree if and only if ${\rm \underline{Hom}}_\Lambda(\Omega^n M, \Omega^n D\Lambda) = 0$. The torsionfreeness criterion can then be recast over Artin algebras as follows: 

\begin{cor}[Torsionfreeness criterion for Artin algebras] Let $\Lambda$ be an Artin algebra. Assume that for each simple module $S$, there is an $n_S \in \N$ such that	${\rm \underline{Hom}}_\Lambda(\Omega^{n_S} S,\ \! \Omega^{n_S} D\Lambda) = 0$. Then $\dell \Lambda \leq \sup_S n_S < \infty$. 
\end{cor}

\subsection{The delooping level of standard classes of algebras}\label{subsectiondeloopinglevel} 

We now study the invariant $\dell \Lambda$ for well-understood classes of algebras. To start, the following characterisation of the finitistic dimension zero case is essentially due to Bass (compare with the dual statement of Theorem 6.3 (4) in \cite{Bass60}). 
\begin{prop}[Bass's Theorem]\label{basstheorem} Let $\Lambda$ be Noetherian semiperfect. The following are equivalent: 
	\begin{enumerate}[i)] 
	\item $\depth \Lambda = 0$. 
	\item $\findim \Lambda^{op} = 0$.
	\item $\dell \Lambda = 0$.
\end{enumerate}
If $\Lambda$ is Artinian, this is also equivalent to: 
	\begin{enumerate}[i)]
		\setcounter{enumi}{3}
	\item $\Findim \Lambda^{op} = 0$.
\end{enumerate}
\end{prop}
\begin{proof} We have $\depth \Lambda \leq \findim \Lambda^{op} \leq \dell \Lambda$, and $\depth \Lambda \leq \Findim \Lambda^{op} \leq \dell \Lambda$ in the Artinian case. It's enough to show that $\depth \Lambda = 0$ implies $\dell \Lambda = 0$. When $\depth \Lambda = 0$, we have $S^* \neq 0$ for every simple $S$ and so $S \hookrightarrow \Lambda_\Lambda$, which shows that $\dell \Lambda = \sup_S \dell S = 0$ as we wanted. 
\end{proof}

\begin{exmp} Let $\Lambda$ be a local Artinian ring. Then $\dell \Lambda = 0$. 
\end{exmp} 

Another situation where everything is well-understood is the Gorenstein case, due to results of Angeleri-H\"ugel--Herbera--Trlifaj. 
\begin{exmp}[Gorenstein rings]\label{fullgorensteinexample} Let $\Lambda$ be Gorenstein. Then we have equalities 
	\begin{align*}
		\findim \Lambda^{op} = \findim \Lambda = \Findim \Lambda^{op} = \Findim \Lambda = \dell \Lambda^{op} = \dell \Lambda = \idim \Lambda_\Lambda.
	\end{align*} 

	Indeed $\Sigma^n\Omega^n S$ is Gorenstein-projective for every $n \geq \idim \Lambda_\Lambda$ and $S$ simple and so $\dell \Lambda \leq \idim \Lambda_\Lambda$ by the torsionfreeness criterion. Proposition \ref{inequalities} then gives $\findim \Lambda^{op} \leq \dell \Lambda \leq \idim \Lambda_\Lambda$. By \cite[Thm. 3.2]{AHHT06} we have $\findim \Lambda^{op} = \Findim \Lambda^{op} = \idim \Lambda^{op}_{\Lambda^{op}} = \idim \Lambda_\Lambda$, and so $\findim \Lambda^{op} = \dell \Lambda = \idim \Lambda_\Lambda$. The remaining equalities follow by symmetry. 
\end{exmp}
The depth of Gorenstein rings is more subtle and will be considered at the end of Section \ref{subsectionmaintheorem}.

Next, we consider two related invariants which have been used to bound the finitistic dimension, namely the notion of ``strongly redundant image'' of Fuller--Wang \cite{FW93} and the repetition index of Goodearl--Huisgen-Zimmermann \cite{GHZ98}. Both notions measure the amount of similarity in indecomposable summands of syzygies of $M$. 

\begin{defn}[Fuller--Wang \cite{FW93}] A module $M \in \modsf \Lambda$ has strongly redundant image from an integer $n$ if $\Omega^n M = \bigoplus_{i \in I} M_i$ where each $M_i$ occurs as a summand of $\Omega^{n+t_i} M$ for some $t_i > 0$. 
\end{defn}

\begin{defn}[Huisgen--Zimmermann \cite{GHZ98}] A module $M \in \modsf \Lambda$ has repetition index at most $n$ if $\Omega^n M = \bigoplus_{i \in I} M_i$ where each $M_i$ occurs as a summand of $\Omega^{n+t_i} M$ for infinitely many $t_i > 0$. 
\end{defn} 
Note that we consider syzygy modules defined only up to direct sum with projectives, and so projective summands of $\Omega^n M$ are considered to repeat in higher syzygies as needed.  

To put terminology on the same footing, we let the redundancy index ${\rm red}\ \! M$ be the least $n$ from which $M$ has strongly redundant image. Likewise we write ${\rm rep}\ \! M$ for the repetition index. We immediately have ${\rm red}\ \! M \leq {\rm rep}\ \! M$. It is easy to see that both invariants bound the delooping level. 

\begin{prop} Let $\Lambda$ be Noetherian semiperfect. Then every $M \in \modsf \Lambda$ satisfies $\dell M \leq {\rm red}\ \! M \leq {\rm rep}\ \! M$. 
\end{prop}
\begin{proof}
	Assume $n = {\rm red}\ \! M < \infty$. Since $\Lambda$ is Noetherian, $\Omega^n M = M_1 \oplus \dots \oplus M_k$ has a complete decomposition into indecomposables, and each $M_i$ is a summand of $\Omega^{n+t_i}$ for some $t_i > 0$ by assumption. We then see that $\Omega^n M$ is a summand of 
	\begin{align*}
		\Omega^{n+1}(\Omega^{t_1 - 1} M \oplus \dots \oplus \Omega^{t_k - 1} M) 
	\end{align*}
	and so $\dell M \leq n$. 
\end{proof}

Natural examples are as follows. We say that an Artinian ring $\Lambda$ is $n$-syzygy finite for some $n \geq 1$ whenever the full subcategory of summands of $n$-th syzygies ${\sf add}\left(\Omega^n(\stmodsf \Lambda)\right) \subseteq \stmodsf \Lambda$ has finitely many indecomposable objects. In this case the repetition index of every $M \in \stmodsf \Lambda$ is finite, as the sequence 
\begin{align*}
	\bigcup_{i \geq n} {\sf add}(\Omega^i M) \supseteq 
	\bigcup_{i \geq n+1} {\sf add}(\Omega^i M) \supseteq 
	\bigcup_{i \geq n+2} {\sf add}(\Omega^i M) \supseteq \dots
\end{align*} 
must eventually stabilise and $M$ has repetitive syzygy summands from then on. 
\begin{exmp} Let $\Lambda$ be Artinian and $n$-syzygy finite. Then $\dell \Lambda < \infty$. 
\end{exmp}

Artinian monomial path algebras over a field are $2$-syzygy finite by \cite{HZ91}.  
\begin{exmp}Let $\Lambda$ be an Artinian monomial path algebra. Then $\dell \Lambda < \infty$. 
\end{exmp}

When $\Lambda$ is a radical square zero Artinian ring (i.e. satisfies $J^2 = 0$), it is easy to see that kernels of projective covers are semisimple and so $\Lambda$ is $1$-syzygy finite. 

\begin{exmp}Let $\Lambda$ be a radical square zero Artinian ring. Then $\dell \Lambda < \infty$. 
\end{exmp}

\section{The Auslander--Bridger grade conditions}\label{sectionAuslanderBridgerconditions} We now investigate when the inequalities 
\begin{align*}
	\depth \Lambda \leq \findim \Lambda^{op} \leq \dell \Lambda	 
\end{align*}
become equalities. We will give sufficient criterion in terms of a grade condition introduced by Auslander--Bridger. 
Recall that we defined $\grade M := \inf {\rm Ext}^*_\Lambda(M, \Lambda)$. 

\begin{defn} A module $M \in \modsf \Lambda$ satisfies the $n$-th Auslander--Bridger condition (or $n$-th grade condition) if $\grade {\rm Ext}^i_\Lambda(M, \Lambda) \geq i$ for each $1 \leq i \leq n$. 
\end{defn}
To put these conditions in context, we mention a theorem of Bass and Auslander--Bridger in commutative algebra. 

\begin{exmp}[{\cite[Prop. 4.21]{AB69}}, \cite{Bass63} for $n = \infty$]\label{AuslanderBridgerproposition} Let $R$ be a commutative Noetherian ring. The following are equivalent: 
	\begin{enumerate}[i)]
	\item $\grade {\rm Ext}^i_R(M, R) \geq i$ for all $M \in \modsf R$ and $1 \leq i \leq n$.  
	\item $R_{\mathfrak{p}}$ is Gorenstein for each prime $\mathfrak{p}$ with $\depth R_{\mathfrak{p}} < n$. 
\end{enumerate} 
\end{exmp}

In the noncommutative case, we also have a result of Auslander--Reiten.  
\begin{exmp}[{\cite[Thm. 0.1]{AR96}}]\label{auslanderreitentheorem} Let $\Lambda$ be a Noetherian ring. The following are equivalent: 
	\begin{enumerate}[a)] 
	\item $\grade {\rm Ext}^i_\Lambda(M, \Lambda) \geq i$ for all $M \in \modsf \Lambda$ and $1 \leq i \leq n$. 
	\item The terms of the minimal injective resolution of $\Lambda^{op}$ 
		\begin{align*}
			0 \to \Lambda^{op}_{\Lambda^{op}} \to I^0 \to I^1 \to \cdots \to I^{i-1} \to I^i \to I^{i+1} \to \cdots	
		\end{align*}
		satisfy $\fdim I^i \leq i+1$ for all $0 \leq i \leq n-1$. 
\end{enumerate}
\end{exmp}

The theorems of Bass, Auslander--Bridger and Auslander--Reiten show that imposing the grade conditions on all modules $M \in \modsf \Lambda$, up to fixed $n \in \N \cup \{ \infty \}$, can be a significant constraint on a ring. 

We can instead ask whether a given module $M$ satisfies the $n$-th grade conditions for some $n \in \N$. Of these conditions, the first non-vacuous one arises for $n = \grade M$. Let us write $j_M := \grade M$ for short. 
\begin{defn} We say that a module $M \in \modsf \Lambda$ satisfies the first non-trivial grade condition if $\grade {\rm Ext}^{j_M}_\Lambda(M, \Lambda) \geq j_M$. Equivalently, 
	\begin{align*}
		{\rm Ext}^j_{\Lambda^{op}}({\rm Ext}^{j_M}_\Lambda(M, \Lambda), \Lambda) = 0 \textup{ for } j < j_M.	
	\end{align*} 
\end{defn} 

A natural condition one might ask of a ring is whether its simple modules satisfy the first non-trivial grade condition. The next example suggests this is a significantly weaker constraint, i.e. compare the following with Example \ref{AuslanderBridgerproposition}. 

\begin{exmp}\label{commutativefirstnontrivialcondition} Let $R$ be a commutative Noetherian ring and let $S = R/\m$ be a simple module, with $\m$ the corresponding maximal ideal. The module ${\rm Ext}^{j_S}_R(S, R)$ is a finite length module supported on $\{ \m \}$, and ${\rm Ext}^j_R(-, R)$ vanishes on such modules for $j < j_{S}$. Therefore $S$ always satisfies the first non-trivial grade condition. 
\end{exmp} 

Our motivation for introducing the grade conditions is the spherical filtration theorem of Auslander--Bridger; it turns out that a lot of useful things can be said about $\Sigma^n \Omega^n M$ when $M$ satisfies the $n$-th grade condition. We assume $\Lambda$ is Noetherian semiperfect in what follows, and begin with a definition. 

\begin{defn}[Spherical module] Let $E \neq 0$ and $n \geq 1$. A module $K \in \modsf \Lambda$ is called $n$-spherical of type $E$ if: 
\begin{enumerate}[(a)] 
\item $\pdim K \leq n$, 
\item ${\rm Ext}^i_\Lambda(K, \Lambda) = 0$ for $1 \leq i \leq n-1$, 
\item ${\rm Ext}^n_\Lambda(K, \Lambda) \cong E$. 
\end{enumerate}
If $E = 0$, we say that $K$ is $n$-spherical of type $E$ if it is projective. 
\end{defn}

	\begin{thm}[Spherical Filtration Theorem {\cite[Theorem 2.37]{AB69}}] Assume that $M \in \modsf \Lambda$ satisfy the $n$-th grade condition. Then there is a projective module $P$ and a filtration 
		\begin{align*}
			M_n \subseteq M_{n-1} \subseteq \dots \subseteq M_1 \subseteq M_0 = M \oplus P	 
		\end{align*}
		such that $M_k \simeq \Sigma^k \Omega^k M$ and $M_{k-1}/M_k$ is $k$-spherical of type ${\rm Ext}^k_\Lambda(M, \Lambda)$ for each $1 \leq k \leq n$. Moreover, the dual map $(M \oplus P)^* \twoheadrightarrow M_k^*$ is onto for each $1 \leq k \leq n$. 
	\end{thm}
	\begin{proof}
		This is \cite[Theorem 2.37]{AB69}. The claim about the dual map is shown in the proof, and relies on \cite[Corollary 2.32]{AB69}. 
	\end{proof}

	We will extract a special case in suitable form. We need a lemma.  
	\begin{lem}[{\cite[Lemma 2.34 c)]{AB69}}]\label{vanishinglemma} Let $M$ satisfy the $n$-th grade condition. Then we have ${\rm Ext}^i_\Lambda(\Sigma^k \Omega^k M, \Lambda) = 0$ for $1 \leq i \leq k$ and all $1 \leq k \leq n$.  
	\end{lem} 

	Next, recall that a morphism of modules is essential if it is non-zero in the stable category. From the spherical filtration theorem, we obtain: 
	\begin{cor}\label{sphericalfiltrationcorollary} Let $S$ be a simple module of grade $j_S$ with $1 \leq j_S < \infty$, and assume that $S$ satisfies the $j_S$-th grade condition. Then there is a short exact sequence of the form 
		\begin{align*}
			0 \to \Sigma^{j_S} \Omega^{j_S} S \to S \oplus Q \xrightarrow{\alpha} K \to 0	 
		\end{align*} 
		for $Q$ some projective module and $K$ a $j_S$-spherical module of type $E := {\rm Ext}^{j_S}_\Lambda(S, \Lambda)$. Moreover: 
		\begin{enumerate}[i)] 
		\item The dual map $(S \oplus Q)^* \twoheadrightarrow (\Sigma^{j_S} \Omega^{j_S} S)^*$ is onto. 
		\item The map $\alpha$ is essential.
	\end{enumerate}
	\end{cor}
	\begin{proof} Applying the spherical filtration theorem for $n = j_S$, we find a projective $P$ and a filtration
		\begin{align*} 
			S_{j_S} \subseteq S_{j_S-1} \subseteq \dots \subseteq S_1 \subseteq S_0 = S \oplus P
		\end{align*} 
		with $S_{j_S} \simeq \Sigma^{j_S} \Omega^{j_S} S$ and $S_{j-1}/S_j$ a $j$-spherical module of type ${\rm Ext}^j_\Lambda(S, \Lambda)$ for each $1 \leq j \leq j_S$. For each $j < j_S$ we have ${\rm Ext}^{j}_\Lambda(S, \Lambda) = 0$ by assumption and so $S_{j-1}/S_j$ is projective. All but the last inclusions then split, and give stable equivalences 
		\begin{align*}
			S_{j_S-1} \simeq S_{j_S - 2} \simeq \dots \simeq S_1 \simeq S_0 = S \oplus P.
		\end{align*} 
		In particular, there exists projective modules $P', P''$ such that 
		\begin{align*}
		S_{j_S-1} \oplus P'' \cong S_0 \oplus P' = S \oplus Q
		\end{align*} 
		where we set $Q := P \oplus P'$. Consider the monomorphism 
		\begin{align*}
			\iota \oplus {\rm id}: S_{j_S} \oplus P'' \hookrightarrow S_{j_S-1} \oplus P'' \cong S \oplus Q.
		\end{align*}
		This gives rise to a short exact sequence 
		\begin{align*}
			0 \to S_{j_S} \oplus P'' 
			\to S \oplus Q \xrightarrow{\alpha} K \to 0
		\end{align*}
		where $K$ is given by 
\begin{align*}
K &\cong (S \oplus Q)/ (S_{j_S} \oplus P'')\\
&\cong (S_{j_S - 1} \oplus P'') / (S_{j_S} \oplus P'') \\
&\cong S_{j_S - 1} / S_{j_S} 
\end{align*}
which is a $j_S$-spherical module of type ${\rm Ext}^{j_S}_\Lambda(S, \Lambda)$. Taking $\Sigma^{j_S} \Omega^{j_S} S := S_{j_S} \oplus P''$ as stable representative, we obtain the short exact sequence of the main claim. 

		The claim i) for $Q$ follows immediately from the analogous claim for $P$ in the spherical filtration theorem. For ii), to show that $\alpha$ is essential, note that the dual sequence is exact by i) and that ${\rm Ext}^i_\Lambda(\Sigma^{j_S} \Omega^{j_S} S, \Lambda) = 0$ for $1 \leq i \leq j_S$ by Lemma \ref{vanishinglemma}. The long exact sequence then gives the isomorphism 
		\begin{align*}
			\alpha^*: {\rm Ext}^{j_S}_\Lambda(K, \Lambda) \xrightarrow{\cong} {\rm Ext}^{j_S}_\Lambda(S \oplus Q, \Lambda) = {\rm Ext}^{j_S}_\Lambda(S, \Lambda).
		\end{align*}
		Since ${\rm Ext}^{j_S}_\Lambda(S, \Lambda) \neq 0$ by assumption and $j_S \geq 1$ we see that $\alpha$ is essential. 
	\end{proof}

	Finally, the short exact sequence of Corollary \ref{sphericalfiltrationcorollary} 
\begin{align*}
	0 \to \Sigma^{j_S} \Omega^{j_S} S \to S \oplus Q \xrightarrow{\alpha} K \to 0	 
\end{align*}
will allow us to access the structure of $\Sigma^{j_S} \Omega^{j_S} S$, and to deduce when it is torsionfree. 

\subsection{Main theorem}\label{subsectionmaintheorem} 

Before stating the main results, let us first state a lemma which essentially amounts to an aesthetic improvement on our results. We have mentioned that the counit map 
\begin{align*}
\Sigma^n \Omega^n M \to M	
\end{align*}
always becomes a retract after applying $\Omega^n(-)$. A stronger statement is true in the presence of the grade conditions. 

\begin{lem}[{\cite[Lemma 2.34 a)]{AB69}}]\label{aestheticlemma} Let $M$ satisfy the $n$-th grade condition. Then the counit $\Sigma^n \Omega^n M \to M$ becomes a stable equivalence $\Omega^n \Sigma^n \Omega^n M \xrightarrow{\simeq} \Omega^n M$ after applying $\Omega^n(-)$. 
\end{lem} 
\begin{proof}
\cite[Lemma 2.34 a)]{AB69} shows that $\Omega^n(\varepsilon_n): \Omega^n \Sigma^n \Omega^n M \to \Omega^nM$ induces an isomorphism of functors ${\rm Ext}^{1}(\Omega^n M, -) \xrightarrow{\cong} {\rm Ext}^1(\Omega^n \Sigma^n \Omega^n M, -)$, and this implies $\Omega^n(\varepsilon_n)$ is a stable equivalence by a result of Eckmann-Hilton \cite[Prop. 1.41]{AB69}. 
\end{proof}

We now state the key result of this section. Recall that $S$ satisfies the first non-trivial grade condition if 
\begin{align*}
{\rm Ext}^j_{\Lambda^{op}}({\rm Ext}^{j_S}_\Lambda(S, \Lambda), \Lambda) = 0 \textup{ for } j < j_S. 
\end{align*}
If $j_S = \grade S = \infty$ we interpret this condition to hold vacuously. 

\begin{prop}\label{keyproposition} Let $\Lambda$ be a Noetherian semiperfect ring. Let $S \in \modsf \Lambda$ be a simple module of finite grade $j_S < \infty$, and assume that $S$ satisfies the first non-trivial grade condition. Then $\Sigma^{j_S} \Omega^{j_S} S$ is torsionfree, and the unit map $\eta_{j_S + 1}$ gives rise to a stable equivalence 
	\begin{align*} 
		\Omega^{j_S} S \xrightarrow{\simeq} \Omega^{j_S+1} \Sigma^{j_S + 1} \Omega^{j_S} S. 
	\end{align*}
	In particular, $\dell S \leq j_S$. 
\end{prop} 
\begin{proof} We show that $\Sigma^{j_S} \Omega^{j_S} S$ is torsionfree; recall that a module is torsionfree if and only if it is a syzygy module, or equivalently a submodule of a projective. 

	When $j_S = 0$ then $S^* \neq 0$ by definition, so that $S \hookrightarrow \Lambda_\Lambda$ and $S$ is torsionfree. When $j_S \geq 1$, Proposition \ref{sphericalfiltrationcorollary} gives a short exact sequence 
		\begin{align*} 
			0 \to \Sigma^{j_S} \Omega^{j_S} S \to S \oplus Q \xrightarrow{\alpha} K \to 0 
		\end{align*} 
		with $Q$ projective and $\alpha$ essential. We claim that the composite 
		\begin{align*}
		\Sigma^{j_S} \Omega^{j_S} S \hookrightarrow  S \oplus Q \twoheadrightarrow Q
		\end{align*}
		is a monomorphism. If not, then the simple module $S = S \oplus 0 \subseteq S \oplus Q$ intersects $\Sigma^{j_S}\Omega^{j_S} S$ non-trivially and must be contained in it. In particular $\alpha(S) = 0$, contradicting that $\alpha$ is essential. Hence we have obtained an embedding $\Sigma^{j_S} \Omega^{j_S} S \hookrightarrow Q$, and therefore $\Sigma^{j_S} \Omega^{j_S} S$ is torsionfree. 

		The remainder of the proof is as in Theorem \ref{theoremtwo}: the unit map $\eta_1$ is a stable equivalence 
		\begin{align*}
			\Sigma^{j_S} \Omega^{j_S} S \xrightarrow{\simeq} \Omega \Sigma(\Sigma^{j_S} \Omega^{j_S} S)
		\end{align*}
		to which we apply $\Omega^{j_S}(-)$ and combine with the unit map $\eta_{j_S}$ 
\begin{align*}
	\Omega^{j_S} S \xrightarrow{\simeq} \Omega^{j_S} \Sigma^{j_S} \Omega^{j_S} S \xrightarrow{\simeq} \Omega^{j_S+1} \Sigma^{j_S+1} \Omega^{j_S} S.
\end{align*}
The first map $\eta_{j_S}$ is a stable equivalence by Lemma \ref{aestheticlemma} and the composition equals $\eta_{j_S+1}$, which is then a stable equivalence. This is what we had to show. 
	\end{proof}

	We can apply this immediately to computing delooping levels. Recall that by Proposition \ref{inequalities}, we have inequalities 
	\begin{align*}
		\depth \Lambda \leq \findim \Lambda^{op} \leq \dell \Lambda	
	\end{align*} 
	which improve in the Artinian case to
	\begin{align*} 
		\depth \Lambda \leq \findim \Lambda^{op} \leq \Findim \Lambda^{op} \leq \dell \Lambda.
	\end{align*} 
	As main result of this paper, we have: 
\begin{thm}\label{theoremthree} Let $\Lambda$ be a Noetherian semiperfect ring. Assume that the first non-trivial grade condition holds for each simple $S \in \modsf \Lambda$. Then 
	\begin{align*} 
		\depth \Lambda = \findim \Lambda^{op} = \dell \Lambda.	
	\end{align*}
	If $\Lambda$ is furthermore Artinian, then the stronger equality holds:
	\begin{align*}
		\depth \Lambda = \findim \Lambda^{op} = \Findim \Lambda^{op} = \dell \Lambda.	
	\end{align*}
\end{thm}
\begin{proof} If $\depth \Lambda = \infty$ the result is immediate, and so we assume that $\depth \Lambda < \infty$. In this case $j_S = \grade S < \infty$ for each simple $S$, and Proposition \ref{keyproposition} shows that $\dell S \leq \grade S$. Taking supremum over simple modules gives $\dell \Lambda \leq \depth \Lambda$, and the result follows. 
\end{proof} 

As a consequence of the equality $\findim \Lambda^{op} = \Findim \Lambda^{op}$, we obtain a positive case of the first finitistic dimension conjecture for Artin algebras. 

\begin{cor}\label{firstfinitisticdimensionconjecture} Let $\Lambda$ be an Artinian ring. Assume that the first non-trivial grade condition holds for each simple $S \in \modsf \Lambda$. Then the first finitistic dimension conjecture holds for $\Lambda^{op}$. 
\end{cor}

\begin{exmp}[Rings of depth zero] Let $\Lambda$ be Noetherian semiperfect, and assume that $\depth \Lambda = 0 $. The first non-trivial grade conditions for simples are then vacuous, and Theorem \ref{theoremthree} gives 
	\begin{align*} 
		\depth \Lambda = \findim \Lambda^{op} = \dell \Lambda = 0
	\end{align*} 
	and additionally $\Findim \Lambda^{op} = 0$ in the Artinian case. Hence we recover Bass's Theorem (Prop. \ref{basstheorem}) as a limiting case.   

\end{exmp}

\begin{exmp}Let $R$ be a commutative Noetherian semiperfect ring. We saw in Example \ref{commutativefirstnontrivialcondition} that simple modules always satisfy the first non-trivial grade conditions, and so Theorem \ref{theoremthree} holds for $R$. 
\end{exmp}

\begin{exmp}\label{depthgorenstein}Let $\Lambda$ be a Noetherian semiperfect ring. Assume that the terms of the minimal injective resolution of $\Lambda^{op}$ 
		\begin{align*}
			0 \to \Lambda^{op}_{\Lambda^{op}} \to I^0 \to I^1 \to \cdots \to I^{i-1} \to I^i \to I^{i+1} \to \cdots	
		\end{align*}
		satisfy $\fdim I^i \leq i+1$ for all $i \geq 0$. Then by Auslander--Reiten's Theorem (Example \ref{auslanderreitentheorem}) the $n$-th grade conditions hold for all $M \in \modsf \Lambda$ and all $n \geq 1$, and therefore Theorem \ref{theoremthree} holds for $\Lambda$. 
\end{exmp}

Auslander has studied the stronger condition that $\fdim I^i \leq i$ for all $i \geq 0$ in the situation of Example \ref{depthgorenstein}, and has shown that this condition is closed under taking opposite ring. Rings satisfying this condition are called Auslander rings. If $\Lambda$ is additionally Gorenstein, we call such a ring Auslander--Gorenstein. Putting together Theorem \ref{theoremthree} with the results already mentioned for Gorenstein rings (Example \ref{fullgorensteinexample}), we obtain: 

\begin{exmp}\label{auslandergorenstein} Let $\Lambda$ be Auslander-Gorenstein. Then $\Lambda$ satisfies Theorem \ref{theoremthree}, and we have equalities 
	\begin{align*} 
		&\depth \Lambda =  \findim \Lambda^{op} = \Findim \Lambda^{op} = \dell \Lambda = \idim \Lambda_\Lambda. 
	\end{align*}
\end{exmp} 

When $\Lambda$ is merely a Gorenstein ring, we always have (in)equalities 
\begin{align*} 
	\depth \Lambda \leq \findim \Lambda^{op} = \Findim \Lambda^{op} = \dell \Lambda = \idim \Lambda_\Lambda.	
\end{align*}
Iwanaga \cite{Iw80} has given an example of a Gorenstein Artin algebra which is not Auslander-Gorenstein. Over general Gorenstein rings it isn't clear when the last equality in fact holds. 

\begin{quest} Let $\Lambda$ be a Noetherian semiperfect Gorenstein ring. When does $\depth \Lambda = \idim \Lambda_\Lambda$? 
\end{quest}

\subsection{Higher Auslander--Bridger sequences}\label{subsectionhigherAuslanderBridgersequences} 
In this section, for any $M \in \modsf \Lambda$ of grade $n$ we further investigate the relationship between the first non-trivial grade condition 
\begin{align*}
	{\rm Ext}^j_{\Lambda^{op}}({\rm Ext}^{n}_\Lambda(M, \Lambda), \Lambda) = 0 \textup{ for } j < n 
\end{align*}

and the torsionfreeness properties of $\Sigma^n \Omega^n M$. 

When $n = 0$ this amounts to studying $M$ itself. Recall that Auslander--Bridger \cite{AB69} have constructed a natural short exact sequence of the form 
	\begin{align*}
		0 \to {\rm Ext}_{\Lambda^{op}}^1({\rm Tr}M, \Lambda) \to M \xrightarrow{\sigma} M^{**} \to {\rm Ext}_{\Lambda^{op}}^2({\rm Tr} M, \Lambda) \to 0.
	\end{align*} 
	This sequence shows that $M$ is torsionfree if and only if ${\rm Ext}^{1}_{\Lambda^{op}}({\rm Tr}M, \Lambda)=0$, and that $M$ is reflexive if and only if ${\rm Ext}^{i}_{\Lambda^{op}}({\rm Tr}M, \Lambda)=0$ for $i = 1, 2$. 

	When $n \geq 1$ the above sequence becomes somewhat degenerate, as $M^* = 0$ implies that $\sigma = 0$ and $M$ has no chance of being torsionfree nor reflexive. In this situation, what we observe here is that there are ``higher'' analogs of the Auslander--Bridger sequence which give the relevant information for $\Sigma^n \Omega^n M$. This result was also obtained prior by Hoshino--Nishida in \cite{HN02} using similar methods. 

	\begin{prop}\label{propositionreflexivesequence} Let $M \in \modsf \Lambda$ and assume that $\grade M = n$. Then there is an exact sequence of the form 
		\begin{align*}
			0& \to {\rm Ext}^{n-1}({\rm Ext}^n(M, \Lambda), \Lambda)  \\ 
			&\to {\rm Ext}^1({\rm Tr}\Sigma^n \Omega^n M, \Lambda) \to M \xrightarrow{\sigma_n} {\rm Ext}^n({\rm Ext}^n(M, \Lambda), \Lambda) 
			\to {\rm Ext}^2({\rm Tr} \Sigma^n \Omega^n M, \Lambda) \to 0.
		\end{align*}

		The following then hold:
		\begin{enumerate}[i)] 
		\item $\Sigma^n \Omega^n M$ is torsionfree if and only if ${\rm Ext}^{n-1}({\rm Ext}^n(M, \Lambda), \Lambda) = 0$ and the morphism $\sigma_n$ is injective. 
		\item $\Sigma^n \Omega^n M$ is reflexive if and only if ${\rm Ext}^{n-1}({\rm Ext}^n(M, \Lambda), \Lambda) = 0$ and the morphism $\sigma_n$ is bijective. 
	\end{enumerate}
	\end{prop}
	\begin{proof}
		When $n = 0$ this is simply the Auslander--Bridger sequence. Now assume $n \geq 1$; we first use an auxiliary sequence. Note that for any $X \in \modsf \Lambda^{op}$ the Auslander--Bridger sequence of $X$ is spliced from two short exact sequences, the first of which is 
		\begin{align*}
			0 \to {\rm Ext}^1({\rm Tr}X, \Lambda) \to X \to \Omega \Sigma X \to 0.
		\end{align*} 
	Here we use the representative $\Omega \Sigma X = X/tX = {\rm im}(\sigma: X \to X^{**})$. Now, the map $X^{***} \xrightarrow{\sigma^*} X^*$ is split-surjective, and so we see that the dual map $(\Omega \Sigma X)^* \to X^*$ is also surjective and therefore an isomorphism. 

	Now, if we let $X = {\rm Tr} \Omega^{n-1} M$ and use the identities $\Sigma {\rm Tr} \simeq {\rm Tr} \Omega$ and $\Omega {\rm Tr} \simeq {\rm Tr}\Sigma$, applying ${\rm Hom}_{\Lambda^{op}}(-, \Lambda)$ to the above short exact sequence gives a long exact sequence 
		\begin{align*} 
			\xymatrix@C=10pt@R=10pt{ 
				0 \ar[r]& (\Omega \Sigma{\rm Tr}\Omega^{n-1}M)^* \ar[r]^-{\cong}& ({\rm Tr}\Omega^{n-1} M)^* \ar[r]^-{0}& {\rm Ext}^n(M, \Lambda)^*  \\ 
				\ar[r]& {\rm Ext}^1({\rm Tr}\Sigma \Omega^{n}M, \Lambda) \ar[r]& {\rm Ext}^1({\rm Tr}\Omega^{n-1}M, \Lambda) \ar[r]& {\rm Ext}^1({\rm Ext}^n(M, \Lambda), \Lambda) \\ 
				\ar[r]& {\rm Ext}^2({\rm Tr}\Sigma \Omega^{n}M, \Lambda) \ar[r]& {\rm Ext}^2({\rm Tr}\Omega^{n-1}M, \Lambda) \ar[r]&\cdots \\ 
				& \cdots \ar[r]& {\rm Ext}^{n-1}({\rm Tr}\Omega^{n-1}M, \Lambda) \ar[r]& {\rm Ext}^{n-1}({\rm Ext}^n(M, \Lambda), \Lambda) \\ 
				\ar[r]& {\rm Ext}^1({\rm Tr}\Sigma^{n}\Omega^{n}M, \Lambda) \ar[r]& {\rm Ext}^{1}({\rm Tr}\Sigma^{n-1}\Omega^{n-1}M, \Lambda) \ar[r]& {\rm Ext}^{n}({\rm Ext}^n(M, \Lambda), \Lambda) \\ 
				\ar[r]& {\rm Ext}^{2}({\rm Tr}\Sigma^{n}\Omega^{n} M, \Lambda) \ar[r]& {\rm Ext}^{2}({\rm Tr}\Sigma^{n-1}\Omega^{n-1}M, \Lambda) \ar[r]& \cdots 
	 }
		\end{align*} 

	As $\grade M = n$, the module ${\rm Tr}M$ is $(n-1)$-torsionfree and Proposition \ref{towerproperties} vi) gives stable equivalences 
		\begin{align*}
			\Sigma^{n-1}\Omega^{n-1} M \xrightarrow{\simeq} \cdots \xrightarrow{\simeq} \Sigma \Omega M \xrightarrow{\simeq} M.	 
		\end{align*}
		Moreover as $n \geq 1$, the Auslander--Bridger sequence gives ${\rm Ext}^1({\rm Tr}M, \Lambda) \xrightarrow{\cong} M$ and ${\rm Ext}^2({\rm Tr}M, \Lambda) = 0$, from which we derive the identities 
		\begin{align*}
			&{\rm Ext}^{1}({\rm Tr}\Sigma^{n-1}\Omega^{n-1}M, \Lambda) \cong M	\\ 
			&{\rm Ext}^{2}({\rm Tr}\Sigma^{n-1}\Omega^{n-1}M, \Lambda) = 0. 
		\end{align*}
We then let the morphism $\sigma_n$ be the composite 
\begin{align*}
	\sigma_n: M \xrightarrow{\cong} {\rm Ext}^1({\rm Tr}\Sigma^{n-1} \Omega^{n-1}M, \Lambda) \to {\rm Ext}^n({\rm Ext}^n(M, \Lambda), \Lambda). 
\end{align*}

Finally, when $n = 1$, substituting these identities in the long exact sequence gives us an exact sequence of the stated form. For $n \geq 2$ we require the additional vanishing ${\rm Ext}^{n-1}({\rm Tr}\Omega^{n-1}M, \Lambda) = 0$. Now, $\grade M = n$ gives us that ${\rm Tr}\ \! \Omega M$ is $(n-2)$-torsionfree, and so we have a stable equivalence 
		\begin{align*}
			\Sigma^{n-2} \Omega^{n-1} M \xrightarrow{\simeq} \Omega M.
		\end{align*}
		It follows that 
		\begin{align*}
			{\rm Ext}^{n-1}({\rm Tr}\Omega^{n-1}M, \Lambda) &= 	{\rm Ext}^1({\rm Tr}\Sigma^{n-2} \Omega^{n-1}M, \Lambda) \\
			&= {\rm Ext}^1({\rm Tr}\Omega M, \Lambda)\\ 
			&= 0
		\end{align*} 
		with the last vanishing since the module $\Omega M$ is always torsionfree. This finishes the construction of the exact sequence, and the last statements follow immediately. 
	\end{proof}

	We record some immediate consequences. 
	\begin{cor}\label{corollaryreflexivesequence} Let $M \in \modsf \Lambda$ with $\grade M = n$. Assume that the condition ${\rm Ext}^{n-1}({\rm Ext}^n(M, \Lambda), \Lambda) = 0$ holds, and let $\sigma_n: M \to {\rm Ext}^n({\rm Ext}^n(M, \Lambda), \Lambda)$ be the morphism of Proposition \ref{propositionreflexivesequence}.
		\begin{enumerate}[i)]
		\item If $\sigma_n$ is injective, then $\Omega^n M \to \Omega^{n+1}\Sigma^{n+1} \Omega^n M$ splits in $\stmodsf \Lambda$. 
		\item If $\sigma_n$ is bijective, then $\Omega^n M \to \Omega^{n+2}\Sigma^{n+2} \Omega^n M$ splits in $\stmodsf \Lambda$. 
	\end{enumerate}
	\end{cor} 
	\begin{proof} Use Proposition \ref{propositionreflexivesequence} and recall that when $\Sigma^n \Omega^n M$ is torsionfree then 
		\begin{align*}
			\Sigma^n \Omega^n M \xrightarrow{\simeq} \Omega \Sigma (\Sigma^n \Omega^n M)	 
		\end{align*}
		is a stable equivalence, and when it is reflexive then 
		\begin{align*}
			\Sigma^n \Omega^n M \xrightarrow{\simeq} \Omega^2 \Sigma^2 (\Sigma^n \Omega^n M)	 
		\end{align*}
		is a stable equivalence. In each case, applying $\Omega^n(-)$ as before gives a splitting 
		\begin{align*}
			\Omega^n M \xrightarrow{\eta_n} \Omega^n \Sigma^n \Omega^n M \xrightarrow{\simeq} \Omega^{n+i}\Sigma^{n+i}\Omega^{n}M	
		\end{align*} 
		for $i = 1$ or $2$ as appropriate. 
	\end{proof} 

	Proposition \ref{propositionreflexivesequence} suggests an alternative proof of the key part of Proposition \ref{keyproposition}, namely the statement that $\Sigma^{j_S} \Omega^{j_S} S$ is torsionfree whenever $S$ is a simple module of finite grade $j_S$ satisfying the $j_S$-th grade condition. If the hypothesis 
	\begin{align*}
		{\rm Ext}^{j_S - 1}({\rm Ext}^{j_S}(S, \Lambda), \Lambda) = 0
	\end{align*}
	is assumed to hold, Proposition \ref{propositionreflexivesequence} states that $\Sigma^{j_S} \Omega^{j_S} S$ is torsionfree if and only if the morphism
	\begin{align*}
		\sigma_{j_S}: S \to {\rm Ext}^{j_S}({\rm Ext}^{j_S}(S, \Lambda), \Lambda)	 
	\end{align*}
	is injective. As $S$ is simple, this is equivalent to $\sigma_{j_S} \neq 0$. If we impose the further conditions
	\begin{align*} 
		{\rm Ext}^{j}({\rm Ext}^{j_S}(S, \Lambda), \Lambda) = 0 \textup{ for } j \leq j_S - 2 
	\end{align*} 
	so that $\grade \ \! {\rm Ext}^{j_S}(S, \Lambda) \geq j_S$, then Proposition \ref{keyproposition} already shows that $\Sigma^{j_S} \Omega^{j_S} S$ is torsionfree and so $\sigma_{j_S} \neq 0$ in this case. What follows is a more direct explanation relating the ``higher double dual'' morphism $\sigma_{j_S}$ to the grade conditions. 
	
	Observe that the condition $\grade \ \! {\rm Ext}^{j_S}(S, \Lambda) \geq j_S$ implies the existence\footnote{We have only formally defined $\sigma_{n}$ for modules of grade $n$, but the same construction works implicitly knowing only that grade $\geq n$.} of the map $\sigma_{j_S}$ also for the module $M = {\rm Ext}^{j_S}(S, \Lambda)$, of the form 
	\begin{align*}
		\sigma_{j_S}: {\rm Ext}^{j_S}(S, \Lambda) \to {\rm Ext}^{j_S}({\rm Ext}^{j_S}({\rm Ext}^{j_S}(S, \Lambda), \Lambda), \Lambda). 
	\end{align*}
	We may then compose the two maps as 
	\begin{align*}
		{\rm Ext}^{j_S}(S, \Lambda) \xrightarrow{\sigma_{j_S}} {\rm Ext}^{j_S}({\rm Ext}^{j_S}({\rm Ext}^{j_S}(S, \Lambda), \Lambda), \Lambda) \xrightarrow{ {\rm Ext}^{j_S}(\sigma_{j_S},\ \Lambda)} {\rm Ext}^{j_S}(S, \Lambda). 
	\end{align*}

	When $j_S = 0$, this is simply the classical section-retract
	\begin{align*}
		S^* \xrightarrow{\sigma} S^{***} \xrightarrow{\sigma^*} S^*
	\end{align*}
	and the identity $\sigma^* \circ \sigma = {\rm id}_{S^*}$ is well-known. The following conjecture is the analogous statement for modules of positive grade. 

	\begin{conj} Let $M$ be a module of grade $n$ such that $\grade {\rm Ext}^n(M, \Lambda) \geq n$. Then the maps
	\begin{align*}
		{\rm Ext}^{n}(M, \Lambda) \xrightarrow{\sigma_{n}} {\rm Ext}^{n}({\rm Ext}^{n}({\rm Ext}^{n}(M, \Lambda), \Lambda), \Lambda) \xrightarrow{ {\rm Ext}^{n}(\sigma_{n},\ \Lambda)} {\rm Ext}^{n}(M, \Lambda). 
	\end{align*}
		satisfy ${\rm Ext}^n(\sigma_n, \Lambda) \circ \sigma_n = {\rm id}_{{\rm Ext}^n(M, \Lambda)}$. 
	\end{conj} 
	An immediate corollary is the non-vanishing of $\sigma_n: M \to {\rm Ext}^n({\rm Ext}^n(M, \Lambda), \Lambda)$, which is what we required in the case that $M = S$. 
	
	A proof of the above conjecture would give an alternative proof of the key part of Proposition \ref{keyproposition}, and therefore of the main theorem of this paper (Theorem \ref{theoremthree}), which bypasses the use of the Spherical Filtration Theorem.

	\section{The delooping level of a local Noetherian ring}\label{sectiondeloopinglevellocalnoetherian} 

	We finally arrive at the first result stated in the introduction. 

\begin{thm}\label{commutativelocalnoeth} Let $R = (R, \m, k)$ be a commutative local Noetherian ring. Then we have equalities $\depth R = \findim R = \dell R$. More precisely, when $R$ has depth $d$, the unit map of $(\Sigma^{d+1}, \Omega^{d+1})$ gives rise to a stable equivalence 
	\begin{align*} 
		\Omega^d k \xrightarrow{\simeq} \Omega^{d+1} \Sigma^{d+1} \Omega^d k 
	\end{align*}
	and $d = \depth R$ is the lowest index for which this holds. 
\end{thm} 
\begin{proof}
	Simple modules over commutative Noetherian rings satisfy the first non-trivial grade conditions (Example \ref{commutativefirstnontrivialcondition}). The equality $\depth R = \findim R = \dell R$ then follows from Theorem \ref{theoremthree}, and the second statement is Proposition \ref{keyproposition}. 
\end{proof}

Next, we say that a Noetherian ring $\Lambda = (\Lambda, \m, D)$ is local (or scalar local) if $\Lambda/ \m = D$ is a division ring for $\m := \rad \Lambda$.  Note that $\Lambda$ is automatically semiperfect. Theorem \ref{commutativelocalnoeth} generalises to: 

\begin{thm}\label{noncommutativelocalnoeth} Let $\Lambda = (\Lambda, \m, D)$ be a local Noetherian ring, with the property that ${\rm Ext}^i_{\Lambda}(D, \Lambda)$ and ${\rm Ext}^i_{\Lambda^{op}}(D, \Lambda)$ have finite length for each $i \geq 0$. We have equalities 
	\begin{align*} 
		\depth \Lambda^{op} = \depth \Lambda = \findim \Lambda^{op} = \findim \Lambda = \dell \Lambda^{op} = \dell \Lambda.	
	\end{align*}
	Moreover when $\Lambda$ has finite depth $d$, the unit map gives rise to a stable equivalence 
	\begin{align*}
		\Omega^d D \xrightarrow{\simeq} \Omega^{d+1}\Sigma^{d+1} \Omega^d D
	\end{align*}
	and $d = \depth \Lambda$ is the lowest index for which this holds. 
\end{thm} 
\begin{proof} Without loss of generality we may assume that $\depth \Lambda \leq \depth \Lambda^{op}$, and that $\depth \Lambda < \infty$. Letting $d := \depth \Lambda$, the finite length hypothesis on Ext coupled with $\Lambda$ being local imply the $d$-th grade condition for $D$, and we obtain equalities $\depth \Lambda = \findim \Lambda^{op} = \dell \Lambda$ from Theorem \ref{theoremthree}. We next show that $\depth \Lambda = \depth \Lambda^{op}$, and we will obtain the remaining equalities by symmetry. 

	It suffices to show $\depth \Lambda^{op} \leq \depth \Lambda = d$. Since the $d$-th grade condition holds for $D$, Proposition \ref{keyproposition} shows $\Sigma^d \Omega^d D$ is torsionfree and so ${\rm Ext}^1_{\Lambda^{op}}({\rm Tr}\Sigma^d \Omega^d D, \Lambda) = 0$. Applying the higher Auslander--Bridger sequence (Proposition \ref{propositionreflexivesequence}) gives an exact sequence of the form 
	\begin{align*}
	\cdots \to {\rm Ext}^1_\Lambda({\rm Tr}\Sigma^d \Omega^d D, \Lambda) \to D \xrightarrow{\sigma_d} {\rm Ext}^d_{\Lambda^{op}}({\rm Ext}^d_\Lambda(D, \Lambda), \Lambda) \to \cdots 
	\end{align*} 
	from which we see that $\sigma_d$ is injective, and so $ D \hookrightarrow {\rm Ext}^d_{\Lambda^{op}}({\rm Ext}^d_\Lambda(D, \Lambda), \Lambda)$. In particular, ${\rm Ext}^d_{\Lambda^{op}}({\rm Ext}^d_\Lambda(D, \Lambda), \Lambda) \neq 0$, and since ${\rm Ext}^d_\Lambda(D, \Lambda)$ has finite length this forces ${\rm Ext}^d_{\Lambda^{op}}(D, \Lambda) \neq 0$. We obtain $\depth \Lambda^{op} \leq d = \depth \Lambda$, which is what we needed. Therefore $\depth \Lambda = \depth \Lambda^{op}$. The equalities $\depth \Lambda^{op} = \findim \Lambda = \dell \Lambda^{op}$ then hold by symmetry, and the last statement is Proposition \ref{keyproposition}.  
\end{proof} 
\begin{rem} Wu--Zhang have extended the Auslander--Buchsbaum Formula to the setting of Theorem \ref{noncommutativelocalnoeth}. In particular the four first equalities 
	\begin{align*}
		\depth \Lambda^{op} = \depth \Lambda = \findim \Lambda^{op} = \findim \Lambda 
	\end{align*}
	follow immediately from \cite[Theorem 1.1]{WZ00}. The above theorem then provides an independent proof and extends the result to the delooping level. 
\end{rem}

The symmetry of invariants $\depth \Lambda = \depth \Lambda^{op}$, $\findim \Lambda = \findim \Lambda^{op}$ and $\dell \Lambda = \dell \Lambda^{op}$ for local Noetherian rings is intriguing, and it would be interesting to extend it to the big finitistic dimension. We record this as another question. 

\begin{quest} Let $\Lambda = (\Lambda, \m, D)$ be as in Theorem \ref{noncommutativelocalnoeth}. Do we also have equalities ${\rm Findim}\ \! \Lambda^{op} = {\rm Findim}\ \! \Lambda$? 
\end{quest}

\section{A Cohen-Macaulay inequality}\label{sectioncohenmacaulayinequality} 

We end this paper on a curiosity. 

\begin{thm}\label{CMinequality}Let $\Lambda$ be a Noetherian semiperfect ring, and consider the inequality $\Findim \Lambda^{op} \leq \dell \Lambda$.
	\begin{enumerate}[i)] 
	\item The inequality characterises Cohen-Macaulay rings amongst commutative local Noetherian rings, in which case we have equality. 
	\item The inequality holds for all Gorenstein rings, in which case we have equality.  
	\item The inequality holds for all Artinian rings. 
\end{enumerate} 
\end{thm} 
\begin{proof} i). Let $R$ be commutative local Noetherian. We have $\dell R = \depth R$ and $\Findim R^{op} = \Findim R = \dim R$, and so 
\begin{align*}
	\Findim R^{op} \leq \dell R &\iff \dim R \leq \depth R \\ &\iff \dim R = \depth R \\ &\iff R \textup{ is Cohen-Macaulay.} 
\end{align*}

ii). Let $\Lambda$ be Gorenstein. We have $\Findim \Lambda^{op} = \dell \Lambda = \idim \Lambda_\Lambda$ by Example \ref{fullgorensteinexample} and so equality holds. 

iii). This is part of Proposition \ref{inequalities}. 
\end{proof}

This result is one of few situations where both noncommutative Gorenstein rings and noncommutative Artinian rings fall in line with their commutative counterpart. However, Theorem \ref{CMinequality} also highlights a discrepancy in the Artinian case, and suggests the natural question: 

\begin{quest}\label{CMconjecture} Let $\Lambda$ be Artinian. Do we have $\Findim \Lambda^{op} = \dell \Lambda$? 
\end{quest} 

A positive answer would suggest that the delooping level of a ring serves as the natural replacement for the depth in the noncommutative setting. 

\bibliographystyle{alpha} 
\bibliography{depth} 

\begin{thebibliography}{AHHT06}

\bibitem[AB69]{AB69}
Maurice Auslander and Mark Bridger.
\newblock {\em Stable module theory}.
\newblock Memoirs of the American Mathematical Society, No. 94. American
  Mathematical Society, Providence, R.I., 1969.

\bibitem[AHHT06]{AHHT06}
Lidia Angeleri-H\"ugel, Dolors Herbera, and Jan Trlifaj.
\newblock Tilting modules and {G}orenstein rings.
\newblock {\em Forum Math.}, 18(2):211--229, 2006.

\bibitem[AR96]{AR96}
Maurice Auslander and Idun Reiten.
\newblock Syzygy {M}odules for {N}oetherian {R}ings.
\newblock {\em Journal of Algebra}, 183(1):167 -- 185, 1996.

\bibitem[Bas60]{Bass60}
Hyman Bass.
\newblock Finitistic {D}imension and a {H}omological {G}eneralization of
  {S}emi-{P}rimary {R}ings.
\newblock {\em Transactions of the American Mathematical Society},
  95(3):466--488, 1960.

\bibitem[Bas62]{Bass62}
Hyman Bass.
\newblock Injective {D}imension in {N}oetherian {R}ings.
\newblock {\em Transactions of the American Mathematical Society},
  102(1):18--29, 1962.

\bibitem[Bas63]{Bass63}
Hyman Bass.
\newblock On the ubiquity of {G}orenstein rings.
\newblock {\em Mathematische Zeitschrift}, 82:8--28, 1963.

\bibitem[Bor94]{Bo94}
Francis Borceux.
\newblock {\em Handbook of Categorical Algebra}, volume~1 of {\em Encyclopedia
  of Mathematics and its Applications}.
\newblock Cambridge University Press, 1994.

\bibitem[Buc86]{Bu86}
Ragnar-Olaf Buchweitz.
\newblock Maximal {C}ohen-{M}acaulay {M}odules and {T}ate-{C}ohomology {O}ver
  {G}orenstein {R}ings.
\newblock {\em Manuscript, available at http://hdl.handle.net/1807/16682},
  1986.

\bibitem[FW93]{FW93}
K.R. Fuller and Yong Wang.
\newblock Redundancy in resolutions and finitistic dimensions of noetherian
  rings.
\newblock {\em Communications in Algebra}, 21(8):2983--2994, 1993.

\bibitem[GHZ98]{GHZ98}
K.R. Goodearl and B.~Huisgen-Zimmermann.
\newblock Repetitive resolutions over classical orders and finite dimensional
  algebras.
\newblock {\em Algebras and Modules II, CMS Conference Proceedings}, 24, 1998.

\bibitem[GR71]{GR71}
L.~Gruson and M.~Raynaud.
\newblock Critères de platitude et de projectivité: Techniques de
  ``platification'' d'un module.
\newblock {\em Inventiones mathematicae}, 13:1--89, 1971.

\bibitem[HN02]{HN02}
Mitsuo Hoshino and Kenji Nishida.
\newblock A {G}eneralization of the {A}uslander {F}ormula.
\newblock {\em Representations of Algebras and Related Topics, Proceedings of
  ICRA X, Toronto}, pages 175--186, 2002.

\bibitem[HZ91]{HZ91}
Birge Huisgen~Zimmermann.
\newblock Predicting syzygies over monomial relations algebras.
\newblock {\em Manuscripta mathematica}, 70(2):157--182, 1991.

\bibitem[HZ92]{HZ92}
B.~Huisgen-Zimmermann.
\newblock Homological domino effects and the first {F}initistic {D}imension
  {C}onjecture.
\newblock {\em Inventiones mathematicae}, 1992.

\bibitem[HZ95]{HZ95}
B.~Huisgen-Zimmermann.
\newblock The finitistic dimension conjectures -- a tale of 3.5 decades.
\newblock {\em Abelian Groups and Modules (A. Facchini and C. Menini, Eds.)},
  pages 501--517, 1995.

\bibitem[Iwa80]{Iw80}
Yasuo Iwanaga.
\newblock On rings with finite self-injective dimension {II}.
\newblock {\em Tsukuba J. Math.}, 4(1):107--113, 06 1980.

\bibitem[Jan61]{Ja61}
J.~P. Jans.
\newblock Some generalizations of finite projective dimension.
\newblock {\em Illinois J. Math.}, 5(2):334--344, 06 1961.

\bibitem[Kra98]{Kr98}
Henning Krause.
\newblock Finitistic {D}imension and {Z}iegler {S}pectrum.
\newblock {\em Proceedings of the American Mathematical Society},
  126(4):983--987, 1998.

\bibitem[Mat59]{Ma59}
Eben Matlis.
\newblock Applications of dualities.
\newblock {\em Proc. Amer. Math. Soc.}, 10:659--662, 1959.

\bibitem[Sma98]{Sm98}
Sverre~O. Smal{\o}.
\newblock The supremum of the difference between the big and little finitistic
  dimensions is infinite.
\newblock {\em Proc. Amer. Math. Soc.}, 126:2619--2622, 1998.

\bibitem[WZ00]{WZ00}
Q.S. Wu and J.J. Zhang.
\newblock Some homological invariants of local pi algebras.
\newblock {\em Journal of Algebra}, 225(2):904 -- 935, 2000.

\bibitem[Zak69]{Zaks69}
Abraham Zaks.
\newblock Injective dimension of semi-primary rings.
\newblock {\em Journal of Algebra}, 13(1):73 -- 86, 1969.

\end{thebibliography}

\end{document}